\documentclass[10pt,a4paper,leqno]{amsart}

\usepackage{amsmath}
\usepackage{amsthm}
\usepackage{amsfonts}
\usepackage{amssymb}
\usepackage{graphicx}
\usepackage{amsbsy}
\usepackage{mathrsfs}
\usepackage{dsfont}
\usepackage{bm}
\newcommand{\be}{\begin{equation}}
\newcommand{\ee}{\end{equation}}
\newcommand{\ba}{\begin{array}}
\newcommand{\ea}{\end{array}}
\newcommand{\bea}{\begin{eqnarray}}
\newcommand{\eea}{\end{eqnarray}}
\newcommand{\bee}{\begin{eqnarray*}}
\newcommand{\eee}{\end{eqnarray*}}

\newtheorem{thm}{Theorem}
\newtheorem{lemma}{Lemma}
\newtheorem{prop}{Proposition}
\newtheorem{definition}{Definition}
\newtheorem{cor}{Corollary}
\newtheorem{remark}{Remark}

\newtheorem*{convention*}{Convention}
\newtheorem*{remark*}{Remark}
\newtheorem*{remarks*}{Remarks}

\numberwithin{equation}{section}
\numberwithin{lemma}{section}
\numberwithin{prop}{section}
\numberwithin{definition}{section}
\numberwithin{remark}{section}

\catcode`@=11
\def\section{\@startsection{section}{1}%
  \z@{1.5\linespacing\@plus\linespacing}{.5\linespacing}%
  {\normalfont\bfseries\large\centering}}
\catcode`@=12

\newcommand{\R}{\mathbb{R}}

\newcommand{\C}{\mathbb{C}}
\newcommand{\N}{\mathbb{N}}
\newcommand{\DD}{\Delta}

\newcommand{\RNup}{\mathbb{R}^{N+1}_+}
\renewcommand{\leq}{\leqslant}
\renewcommand{\geq}{\geqslant}

\newcommand{\weakto}{\rightharpoonup}

\newcommand{\eps}{\varepsilon}

\newcommand{\cE}{\mathcal E}

\newcommand{\Hav}{H_{\mathrm{av}}}

\newcommand{\Er}{E^{\mathrm{rad}}}

\def\alphcr{\alpha_*}

\begin{document}

\title[Uniqueness of radial solutions for  the fractional Laplacian]{Uniqueness of radial solutions \\ for the fractional Laplacian}


\author[R. L. Frank]{Rupert L.~Frank}
\address{California Institute of Technology, Department of Mathematics, Pasadena, CA 91125, USA}

\author[E. Lenzmann]{Enno Lenzmann}
\address{University of Basel, Department of Mathematics and Computer Science, Spiegelgasse 1, CH-4051 Basel, Switzerland}

\author[L. Silvestre]{Luis Silvestre}
\address{University of Chicago, Department of Mathematics, 5734 S. University Avenue, Chicago, IL 60637, USA}

\maketitle

\begin{abstract}
We prove general uniqueness results for radial solutions of linear and nonlinear equations involving the fractional Laplacian $(-\Delta)^s$ with $s \in (0,1)$ for any space dimensions $N \geq 1$. By extending a monotonicity formula found by Cabr\'e and Sire \cite{CaSi-10}, we show that the linear equation
$$
(-\Delta)^s u+ Vu = 0 \ \ \mbox{in} \ \ \R^N
$$
has at most one radial and bounded solution vanishing at infinity, provided that the potential $V$ is radial and non-decreasing. In particular, this result implies that all radial eigenvalues of the corresponding fractional Schr\"odinger operator $H=(-\Delta)^s + V$ are simple. Furthermore, by combining these findings on linear equations with topological bounds for a related problem on the upper half-space $\R^{N+1}_+$, we show uniqueness and nondegeneracy of ground state solutions for the nonlinear equation
$$
(-\Delta)^s Q + Q - |Q|^{\alpha} Q = 0 \ \ \mbox{in} \ \ \R^N
$$
for arbitrary space dimensions $N \geq 1$ and all admissible exponents $\alpha >0$. This generalizes the nondegeneracy and uniqueness result for dimension $N=1$ recently obtained by the first two authors in \cite{FrLe-10} and, in particular, the uniqueness result for solitary waves of the Benjamin--Ono equation found by Amick and Toland \cite{AmTo-91}.
\end{abstract}


\section{Introduction and Overview on Main Results}

The purpose of this paper is to derive uniqueness and oscillation results for radial solutions of linear and nonlinear equations that involve the fractional Laplacian $(-\DD)^s$ with $s \in (0,1)$ in arbitrary space dimension $N \geq 1$. In contrast to the situation with local differential operators, it is evident that the theory of ordinary differential equations (ODE)  itself does not provide any means to establish such results. In particular, classical tools such as Sturm comparison, Wronskians, Picard--Lindel\"of iteration, and shooting arguments (which are all purely local concepts) are not at our disposal when analyzing radial solutions in the setting of the nonlocal operator $(-\DD)^s$. Rather, these methods need to be replaced by suitable substitutes based on different arguments that will be developed in the present paper. So far, the lack of such ODE-type results for pseudo-differential operators such as $(-\DD)^s$ has resulted in a list of open problems, conjectures, and spectral assumptions, supported by numerical evidence or verified in some exactly solvable special cases (see, e.\,g., \cite{AmTo-91, Li-04, ChLiOu-06}), as well as some recent more general result for $N=1$ dimension obtained by the first two authors of the present paper in \cite{FrLe-10}. See below for further details and a brief review of the literature on this.

In the present paper, we improve this situation by developing a set of general arguments that establish ODE-type theorems for radial solutions in the fractional setting with $(-\DD)^s$. In fact, most of the results derived here can be extended to a broader class of pseudo-differential operators, as we will indicate in some detail below. However, for the sake of concreteness, we shall consider cases that involve the fractional Laplacian throughout this paper. The main results derived below can be summarized as follows.


\subsection*{Uniqueness of Nonlinear Ground States in $\R^N$ with $N \geq 1$}  

We prove uniqueness and nondegeneracy of ground state solutions $Q \in H^s(\R^N)$ for the nonlinear problem
\be \label{eq:Qfirst}
(-\DD)^s Q + Q - |Q|^{\alpha} Q = 0 \ \ \mbox{in} \ \ \R^N,
\ee
in arbitrary dimension $N \geq 1$ and any admissible exponent $0 < \alpha < \alpha_*$, where the critical exponent $\alpha_*=\alpha_*(s,N)$ is given in \eqref{def:alphacr} below. In particular, this result generalizes (in an optimal way) the recent result by the first two authors in \cite{FrLe-10} in $N=1$ dimension and settles conjecture by Kenig et al.~\cite{KeMaRo-11} and Weinstein \cite{We-87} in any dimension $N \geq 1$. In particular, it generalizes the classical uniqueness result by Amick and Toland \cite{AmTo-91} on the uniqueness of solitary waves for the Benjamin--Ono equation. In the local case when $s=1$, the uniqueness and nondegeneracy of ground states for problem \eqref{eq:Qfirst} was established in a celebrated paper by Kwong \cite{Kw-89} (see also \cite{Co-72,Mc-93}), which provides an indispensable basis for the blowup analysis as well as the stability of solitary waves for related time-dependent equations such as the nonlinear Schr\"odinger equation (NLS) (see e.\,g.~\cite{We-85, MeRa-05}). Further below we will briefly discuss how the results about \eqref{eq:Qfirst} derived below in the fractional setting when $s \in (0,1)$ will be central in the analysis of time-dependent problems such as the generalized Benjamin--Ono equation and critical fractional NLS. In terms of interpolation estimates, the uniqueness result about ground states for equation \eqref{eq:Qfirst} classifies all optimizers of a related fractional Gagliardo--Nirenberg--Sobolev inequality; see below.

In view of the recent result on ground states for \eqref{eq:Qfirst} in $N=1$ dimension, we mention that the higher dimensional case with $N \geq 2$ can be settled here with the help of two new key results on linear equations with $(-\DD)^s$ derived below. In fact, although  the uniqueness of ground states for \eqref{eq:Qfirst} is a global nonlinear result, we shall see (and exploit) an intimate connection to the linear results described next.

\subsection*{Radial Uniqueness for Linear Equations} The fact that ODE techniques are not applicable for the fractional Laplacian in the radial setting has so far resulted in a big gap of results in the spirit of  shooting arguments. In particular, the well-known and essential fact that a radial solution $u=u(r)$ of the linear equation $-\DD u + V u  = 0$ on $\R^N$ (where $V$ obeys a mild regularity condition) satisfies $u(0)=0$ if and only if $u \equiv 0$ has had no counterpart in the fractional setting with $(-\DD)^s$ up to now.  

Here, we shall fill this gap by proving the following result: For radial and non-decreasing potentials $V=V(r)$ in some H\"older class, we show that any radial and bounded solution $u=u(r)$ vanishing at infinity and solving the linear equation
\be \label{eq:Vfirst}
(-\DD)^s u + Vu = 0 \ \ \mbox{in} \ \ \R^N
\ee
satisfies $u(0)=0$ if and only if $u \equiv 0$ on $\R^N$. By linearity of the problem, this is equivalent to saying that equation \eqref{eq:Vfirst} has at most one radial and bounded solution $u(r)$ vanishing at infinity. This result can be seen as a key substitute for shooting arguments in the ODE setting. Hence it turns out to be essential for the understanding of linear and nonlinear radial problems involving the fractional Laplacian.

The proof of this radial uniqueness result involves a sort of an energy argument based on a monotonicity formula for $(-\DD)^s$ (see below). Note that the condition that $V(r)$ is radial and non-decreasing (which physically corresponds to an attractive potential) arises naturally in many situations. In particular, this property of $V$ will naturally be satisfied by the linearized operator $L_+=(-\DD)^s + 1 - (\alpha+1) Q^\alpha$ around the ground states $Q=Q(r) >0$ of problem \eqref{eq:Qfirst}, which is known to be decreasing function in $r=|x|$.

\subsection*{Simplicity of Radial Eigenvalues for Fractional Schr\"odinger Operators}
As a direct consequence of the uniqueness result about \eqref{eq:Vfirst} above, we obtain that all radial eigenvalues of the corresponding fractional Schr\"odinger operator $H=(-\DD)^s + V$ on $\R^N$ are simple. This spectral result, which is a classical fact for $s=1$ by ODE techniques, will be of essential use when deriving the nondegeneracy and then the uniqueness of ground state solution $Q$ for the nonlinear problem \eqref{eq:Qfirst}.

\subsection*{Sturmian Oscillation Estimates} 
For operators $H=(-\DD)^s + V$ as above, we show that its second radial eigenfunction changes its sign exactly once on the half-line $(0,+\infty)$. This result can be regarded as an analogue as the classical oscillation bound for classical Sturm--Liouville problems. In particular, this optimal oscillation result generalizes the result in \cite{FrLe-10} to arbitrary dimension $N \geq 1$. Furthermore, such an oscillation estimate is a central ingredient in the proof of the nondegeneracy of the radial ground state $Q$ for problem \eqref{eq:Qfirst}.

\subsection*{Sketch of Main Ideas}
The proof of the main results mentioned above involve the following three concepts.

\begin{itemize}
\item Topological bounds on the nodal structure of the solutions for equation with $(-\DD)^s$, which are obtained by considering a related problem on the upper half-space $\R^{N+1}_+$ with local differential operators.
\item Homotopic argument by continuation in $s$. That is, we construct solutions $u=u_s$ parameterized by the index $s \in (s_0,1)$ of the fractional Laplacian, with $s_0 \in (0,1)$ fixed. Taking the limit $s \nearrow 1$ yields global information about the branch $u_s$, which can be deduced from the limiting problem at $s=1$ known by classical ODE techniques.   
\item A monotonicity formula for radial solutions for $(-\DD)^s$ with $s \in (0,1)$ on $\R^N$. 
\end{itemize}

We now briefly sketch the three main ideas as follows. How these arguments enter in the individual proofs will be seen further below.  

\subsubsection*{Topological Bounds via Extension to $\R^{N+1}_+$}  Using \cite{CaSi-07} we express the nonlocal operator $(-\DD)^s$ on $\R^N$ with $s \in (0,1)$ as a generalized Dirichlet--Neumann map for a certain elliptic boundary-value problem with local differential operators defined on the upper half-space $\R^{N+1}_+ = \{ (x,t) : x \in \R^N, \, t > 0\}$. That is, given a solution $u=u(x)$ of $(-\DD)^s u = f$ in $\R^N$, we can equivalently consider the dimensionally extended problem for $u=u(x,t)$, still denoted by $u$ for simplicity, which solves 
\be \label{eq:exten_intro}
\left \{ \begin{array}{ll} \mathrm{div} \, (t^{1-2s} \nabla u) = 0 & \quad \mbox{in $\R^{N+1}_+$}, \\
-d_s t^{1-2s} \partial_t u |_{t \to 0} = f&  \quad \mbox{on $\partial \R^{N+1}_+$} . \end{array} \right . 
\ee
Here the positive constant $d_s>0$ is explicitly given by
\be \label{eq:ds}
d_s = 2^{2s-1} \frac{\Gamma(s)}{\Gamma(1-s)}.
\ee 
In particular, the reformulation \eqref{eq:exten_intro} in terms of local differential operators plays a central role when deriving bounds on the number of sign changes for eigenfunctions of fractional Schr\"odinger operators $H=(-\DD)^s + V$. In \cite{FrLe-10} this idea was implemented for $N=1$ dimensions to obtain certain sharp oscillation bounds for eigenfunctions of $H$. However, the case of higher dimensions has been left open, due to the topological fact that $\R^{N+1}_+ \setminus \{ (0,t) : t >0 \}$ is a connected set if $N \geq 2$, whereas it is not connected if $N=1$, which was needed in \cite{FrLe-10}.

In this paper, we will overcome the restriction to $N=1$ and will be able to treat arbitrary space dimension $N \geq 1$, by combining the extension method with a continuation argument in $s$ and a monotonicity formula for radial solutions involving the fractional Laplacians $(-\DD)^s$ with $s\in (0,1)$. See below.  

\subsubsection*{Homotopic Argument by Continuation in $s$} The idea behind this method is to make a continuation  argument with respect to the power $s \in (0,1)$ appearing in $(-\DD)^s$. More precisely, starting from some solution $u_0$ of 
$$
(-\DD)^{s_0} u_0 + f_{s_0}(u_0,x) = 0 \ \ \mbox{in} \ \ \R^N, 
$$
we embed this problem into a suitable family of equations parameterized by $s$. That is, we seek to construct a branch $u_s$ with $s$ close to $s_0$ solving the problem  
$$
(-\DD)^s u_s + f_s(u_s,x) = 0 \ \ \mbox{in} \ \ \R^N.
$$
The local existence and uniqueness (in some function space) for $u_s$ with $u_{s=s_0} = u_{0}$ follows from an implicit function argument, provided the linearization around $u_0$ is invertible, which is the first key step in the argument. The second key point is now to derive suitable a-priori bounds that guarantee that the branch $u_s$ can be extended all the way to $u_s$ as $s \nearrow 1$ converging to a nontrivial limit. Typically, the limiting problem with $s=1$ can be well-understood by ODE methods and, by an open-closed argument, we obtain information for the whole branch $u_s$ for $s \in [s_0,1]$. We will use these arguments in suitable variations when deriving sharp oscillations estimates for the second radial eigenfunction of $H=(-\DD)^s + V$, as well as when we show the global uniqueness of ground states $Q$ for equation \eqref{eq:Qfirst}.
  
\subsubsection*{Monotonicity Formula for $(-\DD)^s$} In a recent and remarkable article \cite{CaSi-10}, Cabr\'e and Sire introduced and exploited a monotone quantity for so-called layer solutions as well as radial solutions of nonlinear equations of the form $(-\DD)^s v = f(v)$ on $\R^N$. Inspired by their work, we formulate a monotone quantity for radial solutions of the linear equation $(-\DD)^s u + Vu =0$ on $\R^N$, by defining 
$$
H(r) = d_s \int_0^{+\infty} \frac{t^{1-2s}}{2} \left \{ u_r^2(r,t) - u_t^2(r,t) \right \}  dt - \frac{1}{2} V(r) u(r)^2 ,
$$ 
where $u=u(r,t)$ with $r=|x|$, $x \in \R^N$, and $t >0$, denotes the extension to the upper-half space $\R^{N+1}_+$ of $u(r)=u(r,0)$ that satisfies problem \eqref{eq:exten_intro} above, where $d_s > 0$ is the same constant as above. In fact, it turns out  (formally at least)  that $H'(r) \leq 0$ holds, provided the potential $V$ is non-decreasing, and hence $H(r)$ is non-increasing under this assumption on $V$. By using the monotonicity of $H(r)$, we  conclude a rather general uniqueness result for radial solutions that vanish at infinity and solve the linear equation $(-\DD)^s u + Vu = 0$ in $\R^N$, saying that $u(0)=0$ implies that $u \equiv 0$ on $\R^N$. For a precise statement, see Theorem \ref{thm:linear1} below.

\subsection*{Review of Known Results}

We briefly recap the results about uniqueness for solutions $u=u(x)$ of problem having the form
\be \label{eq:ugeneral}
\left \{ \begin{array}{l} Ê(-\DD)^s u + f(u,x) = 0  \ \ \mbox{in} \ \ \R^N, \\
u(x) \to 0 \ \ \mbox{as} \ \ |x| \to +\infty. \end{array} \right .
\ee
with $s \in (0,1)$. As usual, $f(u,x)$ stands for some given nonlinear or linear function; e.\,g., a nonlinearity $f(u) = u-u^{\alpha+1}$ with some $\alpha >0$ or a linearity $f(u,x) = V(x) u$ with some given potential $V$.

In the cases of interest, the existence of nontrivial solutions of equation \eqref{eq:ugeneral} can be deduced by standard variational methods adapted to $(-\DD)^s$.  However, in contrast to the classical case $s=1$, very little is known in general (except from a few situations discussed below) about uniqueness of radial solutions for problems of the form \eqref{eq:ugeneral}. Indeed, even the situation of linear $f(u,x) = V(x) u$ has not been understood so far in rudimentary terms.

\subsubsection*{Nonlinear Case} 

For nonlinear problems of the form \eqref{eq:ugeneral}, the known nonperturbative uniqueness results can be summarized as follows. (See \cite{KeMaRo-11, FaVa-13} for  some perturbative uniqueness results when $s$ is close to $1$.)

\begin{itemize}

\item Benjamin--Ono equation: In \cite{AmTo-91} Amick and Toland proved that uniqueness  (up to translations) of the nontrivial solution $Q \in H^{1/2}(\R)$ of 
$$
(-\DD)^{1/2} Q + Q - Q^2 = 0 \ \ \mbox{in} \ \ \R.
$$
In fact, the unique family of solutions $Q(x) = \frac{2}{1+(x-x_0)^2}$ with $x_0 \in \R$ is known in closed form. The proof in \cite{AmTo-91} relies on an intriguing reformulation of the problem in terms of complex analysis and makes also strong use of the fact the nonlinearity is quadratic. However, the methods seem to be rather specific. Therefore, generalizing the proof of Amick and Toland to $(-\DD)^s$ with $s \in (0,1)$, different powers $Q^{\alpha+1}$ with $\alpha \neq 1$, and dimension $N \geq 2$ does not appear to be achievable.

\item In \cite{Li-04,ChLiOu-06} it was shown independently by Y.~Y.~Li and Chen, Li, and Ou that for $s \in (0,N/2)$ and $Q \in L^{\frac{2N}{N-2s}}_{\mathrm{loc}}(\R^N)$ that the (energy-critical) equation
$$
(-\DD)^{s} Q - Q^{\frac{N+2s}{N-2s}} = 0 \ \ \mbox{in} \ \ \R^N
$$
has a unique positive solution $Q(r) > 0$ up to scaling and translation. However, both the uniqueness proofs in \cite{Li-04,ChLiOu-06} make essential use of the fact that this problem exhibits conformal symmetry. Apart from nonexistence results in the energy-subcritical case (see \cite{ChLiOu-05}), an extension of the methods in \cite{Li-04,ChLiOu-06} to prove uniqueness of positive solutions $Q(r)>0$ to a broader class of equations seems out of scope.

\item In \cite{FrLe-10} the first two authors of this paper proved uniqueness and nondegeneracy of ground state $Q \in H^{s}(\R)$ for 
$$
(-\DD)^s Q + Q - Q^{\alpha+1} = 0 \ \ \mbox{in} \ \ \R.
$$
for all $s \in (0,1)$ and all $H^s$-admissible powers $0 < \alpha < \alpha_*$; see \eqref{def:alphacr} below for the definition of $\alpha_*>0$. The proof given in \cite{FrLe-10} develops a general strategy, but it needed the assumption of restricting to $N=1$ space dimension in one crucial step. This basic dimensional obstruction will be overcome in the present paper.
\end{itemize}

\subsubsection*{Linear Case}
For linear problems of the form \eqref{eq:ugeneral}, the question of uniqueness of radial solutions has not been well understood either. A notable but rather specific case arises when $u$ is known to be the ground state of a linear fractional Schr\"odinger operator $H=(-\DD)^s + V$. (By shifting the potential $V$, we can always assume that the lowest eigenvalue $E_1$ of $H$ satisfies $E_1=0$) Then, by standard Perron--Frobenius methods (see Appendix \ref{app:misc}) it follows that $u(x) > 0$ is strictly positive and that the lowest eigenvalue $E_1=0$ is simple.  In particular, we obtain uniqueness of solutions to the linear problem up to multiplicative constants.

However, in many interesting cases linear problems of the form \eqref{eq:ugeneral}, solutions $u$ are {\em not ground states} of some fractional Schr\"odinger $H=(-\DD)^s + V$. In particular, a crucial part in the analysis of blowup and stability solitary wave solutions lead to the study of higher radial eigenfunctions $u$ of some fractional Schr\"odinger operator. Here Perron--Frobenius arguments are of no use and, consequently, the question of uniqueness of such solutions $u$ need to be addressed by novel arguments. Along with this, oscillation estimates for higher radial eigenfunctions $u$ for fractional Schr\"odinger operators $H$ are of central interest.   This will be addressed below too. Let us also mention the oscillation and simplicity results for the spectrum of $\sqrt{-\Delta}$ on the interval $I= (-1,1)$ (with exterior Dirichlet conditions on the complement $I^c$) obtained by Ba{\~n}uelos and Kulczycki \cite{BaKu-04} ; see also \cite{KuKwMaSt-10,Kw-12} for improvements to $(-\DD)^s$ with $s \in [1/2,1)$ in this one-dimensional setting. However, the arguments given in these works do not seem to be extendable to a more general class of operators $H=(-\DD)^s + V$ in general dimension $N \geq 1$ and arbitrary powers $s\in (0,1)$.

\subsection*{Generalizations to other Pseudo-Differential Operators}

The generalization of the main results about linear equation (i.\,e.~Theorems 1 and 2 below) to pseudo-differential operators $L$ beyond the fractional Laplacian on $\R^N$ is feasible, provided they can be regarded as certain Dirichlet--Neumann maps. Important examples (which will be treated in future work) are as follows.
\begin{itemize}
\item $L=(-\DD)^s$ with $s \in (0,1)$ on any open ball $B_R = \{ x \in \R^N : | x | < R \}$ with exterior Dirichlet condition on $B_R^c$. 
\item $L=(-\DD)^s$ with $s\in (0,1)$ on the $N$-dimensional hyperbolic space $\mathbb{H}^N$. 
\item $L=(-\DD + m^2)^{s/2}$ with $s \in (0,1)$ and $m>0$ on $\R^N$, $B_R$, or $\mathbb{H}^N$.
\item $L= (-\DD)^{1/2} \coth (-\DD)^{1/2}$ on $\R$. This pseudo-differential operator arises in the intermediate long-wave equation modeling water waves. 
\end{itemize} 
We refer to \cite{BaKu-04,BaGoSa-12,AbBoFeSa-89} for more details on these operators and their occurrence in probability, geometry and physics.

Regarding the nonlinear main results (i.\,e.~Theorems \ref{thm:nondeg} and \ref{thm:unique} below) about nondegeneracy and uniqueness of ground states $Q$, we remark that an extension to nonlinearities $f(u)$ beyond the pure-power case seems to be a challenging open problem.

\subsection*{Plan of the Paper} 
We organize this paper as follows. In Sections \ref{sec:linear} and \ref{sec:nonlinear}, we state the linear and nonlinear main results, respectively. The proof of Theorem \ref{thm:linear1} (linear uniqueness result) will be given in Section \ref{sec:mono} by using the aforementioned monotonicity formula for $(-\DD)^s$ in the class of radial solutions. The Sections \ref{sec:osc} and  \ref{sec:proofthm2} are devoted to the proof of Theorem \ref{thm:linear2} (linear oscillation result). Finally, in Sections \ref{sec:nondeg} and \ref{sec:unique}, we prove the nonlinear main results; i.\,e.~Theorems \ref{thm:nondeg} and \ref{thm:unique} about nondegeneracy and uniqueness of ground states $Q$ for the nonlinear problem \eqref{eq:Qfirst}.

The appendix contains a variety of technical results (such as regularity and uniform estimates) needed in the main part of this paper.

\subsection*{Notation and Conventions}
Throughout this paper, we employ the common abuse of notation by writing both $f=f(|x|)$ and $f=f(r)$ for any radial functions $f$ on $\R^N$. We use standard notation for $L^p$ and Sobolev spaces and $L^{p}_{\mathrm{rad}}(\R^N)$ denotes the space of radial and square-integrable functions on $\R^N$. For $k \in \mathbb{N}_0$ and $0 < \gamma \leq 1$ the H\"older space $C^{k, \gamma}(\R^N)$ is equipped with the norm $\| u \|_{C^{k,\gamma}} = \sum_{ |\alpha| \leq k} \| \partial_x^\alpha u \|_{L^\infty} + \| u \|_{C^\gamma}$, where $\| u \|_{C^\gamma} = \sup_{x \neq y} \frac{ |u(x)-u(y)|}{|x-y|^\gamma}$. We often write $L^p$ instead of $L^p(\R^N)$ etc.

We employ the following convention for constants in this paper: Unless otherwise explicitly stated, we write 
$$
X \lesssim_{a,b,c,\ldots} Y
$$
to denote that $X \leq C Y$ with some constant $C> 0$ that only depends on the quantities $a,b,c,\ldots$ and the space dimension $N \geq 1$. Moreover, we require that $C>0$ can be chosen uniform if $a,b,c,\ldots$ range in some fixed compact set.

\subsection*{Acknowledgments} R.\,F~thanks Elliott Lieb for useful discussions and Iosif Polterovich for pointing out reference \cite{AlMa-94}. R.\,F.~acknowledges financial support from the NSF grants PHY-1068285, PHY-1347399, and DMS-1363432. E.\,L.~expresses his deep gratitude to J\"urg Fr\"ohlich for his constant support, interest, and inspirations revolving around $(-\DD)^s$. Moreover, E.~L.~acknowledges financial support from the Swiss National Science Foundation (SNF). In addition, R.\,F.~and E.\,L.~thank the Isaac Newton Institute for its kind hospitality in August 2012, where parts of this work were done. L.~S.~acknowledges financial support from the NSF grants DMS-1001629 and DMS-1065979. Finally, the authors thank the anonymous referees for valuable comments.

\section{Linear Main Results} 

\label{sec:linear}

Let $N \geq 1$ and $s \in (0,1)$ be given. We consider the linear equation
\be \label{eq:linear}
 (-\Delta)^s u + V u = 0 \ \ \mbox{in} \ \ \R^N.
\ee
We require that the potential $V: \R^N \to \R$ satisfies the following conditions.

\begin{enumerate}
\item[(V1)] $V=V(|x|)$ is radial and non-decreasing in $|x|$.
\item[(V2)] $V \in C^{0,\gamma}(\R^N)$ for some $\gamma > \max \{ 0, 1-2s \}$.
\end{enumerate}
Throughout the rest of this section, we shall assume that the potential $V$ in equation \eqref{eq:linear} satisfies the above conditions (V1) and (V2). Recall that the conditions $V$ belongs to  $C^{0,\gamma}$ means in particular that $V$ is bounded. 

The first main result establishes the following basic uniqueness result for radial and bounded solutions to \eqref{eq:linear} that vanish at infinity.\footnote{By this, we mean that  the Lebesgue measure of $\{ x \in \R^N : |u(x)|  >  \alpha \}$ is finite for every $\alpha > 0$.}

\begin{thm} \label{thm:linear1}
Let $N \geq 1$ and $s \in (0,1)$. Suppose that $u=u(|x|)$ is a radial and bounded solution of the linear equation \eqref{eq:linear} and that $u$ vanishes at infinity. Then $u(0)=0$ implies that $u \equiv 0$.

Equivalently, we have that the linear equation \eqref{eq:linear} has at most one bounded and radial solution that vanishes at infinity.
\end{thm}

\begin{remark*} {\em
 By regularity estimates (see below), we actually have that $u \in C^{1, \beta}(\R^N)$ holds for some $\beta \in (0,1)$. In particular, the statement $u(0) = 0$ makes  sense. Furthermore, since $u$ vanishes at infinity by assumption, this H\"older estimate implies that $u$ tends to zero pointwise at infinity, i.\,e., it holds that $u(|x|) \to 0$ as $|x| \to +\infty$.  
}

\end{remark*}

An immediate  application of Theorem \ref{thm:linear1} arises for fractional Schr\"odinger operators $H=(-\Delta)^s + V$ with radial potentials $V(r)$ as above. In this case, we can study the restriction of $H$ to the sector of radial functions. Furthermore, by regularity estimates (see below) we conclude that any $L^2$-eigenfunction of $H$ is bounded. Clearly, the shifted potential $V-E$ also satisfies conditions (V1) and (V2) above. Thus we can apply Theorem \ref{thm:linear1} to deduce the simplicity of all eigenvalues of $H$ in the sector of radial functions.

\begin{cor} \label{cor:linear1}
Suppose $N \geq 1$, $s \in (0,1)$, and let $V$ be as above. Consider $H=(-\Delta)^s + V$ acting on $L^2_{\mathrm{rad}}(\R^N)$. Then all  eigenvalues of $H$ are simple. 

In particular, if $\Er_1 \leq  \Er_2 \leq \Er_3  \leq  \cdots <\inf \sigma_{\mathrm{ess}}(H)$ denote (with counting multiplicity) the discrete eigenvalues of $H$ acting on $L^2_{\mathrm{rad}}(\R^N)$, then we have strict inequalities
$$
\Er_1 < \Er_2 < \Er_3 <  \cdots < \inf \sigma_{\mathrm{ess}}(H).
$$
\end{cor}

\begin{remarks*} {\em
1.) Note that Corollary \ref{cor:linear1} also shows the simplicity of any possible embedded eigenvalue of $H$ in the radial sector.  

2.) The simplicity of the lowest eigenvalue $\Er_1 < \Er_2$, together with the strict positivity of the corresponding eigenfunction $\psi_1(x) >0$, follows from standard Perron--Frobenius arguments for $H=(-\Delta)^s + V$ with $0 < s <1$; see also Appendix \ref{app:misc}. However, this classical method is restricted to the case of the lowest eigenvalue of $H$ and it is not applicable to any higher eigenvalue in contrast to the arguments that establish Corollary \ref{cor:linear1}. }
\end{remarks*}

As the last main result for the linear equation \eqref{eq:linear}, we prove the following sharp oscillation estimate for the second eigenfunction of $H$, which provides us with a bound in agreement with Sturm--Liouville theory for ODE.  

\begin{thm} \label{thm:linear2}
Let $H= (-\Delta)^s + V$ be as in Corollary \ref{cor:linear1} above. Suppose that $H$ has at least two radial eigenvalues $\Er_1 < \Er_2 < \mathrm{inf} \, \mathrm{\sigma}_{\mathrm{ess}}(H)$. Let $\psi \in L^2(\R^N)$ denote the radial eigenfunction of $H$ for the second radial eigenvalue $\Er_2$. Then $\psi = \psi(|x|)$ changes its sign exactly once for $|x| = r \in (0,+\infty)$.
\end{thm}

\begin{remarks*} {\em 
1.) By this, we mean that there is some $r_* >0$ such that (after multiplying $\psi$ with $-1$ is necessary) we have 
$$\psi(r) \geq 0 \ \ \mbox{for} \ \ r \in [0,r_*) \quad \mbox{and} \quad \psi(r) \leq 0 \ \ \mbox{for} \ \ r \in [r_*,+\infty), $$ 
and $\psi \not \equiv 0$ on both intervals $[0,r_*)$ and $[r_*, +\infty)$. Note also that $\psi(0) > 0$ by Theorem \ref{thm:linear1}. 


2.) In \cite{FrLe-10} this result was shown for $N=1$ space dimension by using a variational problem posed on the upper half-space $\R^{1+1}_+$. However, carrying over the proof given there to radial solutions in $N \geq 2$ dimensions yields the weaker bound that $\psi$ changes its sign at most twice on $(0,+\infty)$. The reason that the case $N \geq 2$ is different can be traced back to the fact that the set $\R^{N+1}_+ \setminus \{ (0,t) : t > 0 \}$ is  connected when $N \geq 2$, whereas $\R^{1+1}_+ \setminus \{ (0,t) : t >0 \}$ is not connected. Despite this topological complication for $N \geq 2$, we will improve the bound for $\psi$ to the optimal bound as stated in Theorem \ref{thm:linear2}, by further independent arguments based on Theorem \ref{thm:linear1} and a homotopic argument for fractional Schr\"odinger operators $H=(-\DD)^s + V_s$ by continuing the eigenfunction with respect to $s \in (0,1]$.

3.) Note that we require $H \psi = E \psi$ with $E$ strictly below the essential spectrum of $H$. Indeed, we do not expect that $\psi$ changes its sign only once (or even finitely many times) on the half-line in the case when $E > \mathrm{inf} \, \sigma_{\mathrm{ess}}(H)$ is an embedded eigenvalue. By analogy to the classical ODE case when $s=1$, an oscillatory behavior of $\psi$ at infinity is conceivable in this special situation.

4.) In the proof of Theorem \ref{thm:nondeg} below, this sharp oscillation result for the second eigenfunction of $H=(-\DD)^s + V$ will play an essential role. In fact, the second eigenfunction is often of central interest in the linearization of minimizers in variational problems to study their stability behavior. See also the next Section \ref{sec:nonlinear} below.

}
\end{remarks*}

\section{Nonlinear Main Results}

\label{sec:nonlinear}


Let $N \geq 1$ and $s \in (0,1)$  be given. We consider real-valued solutions $Q \in H^s(\R^N)$ of the nonlinear model problem 
\be \label{eq:Q}
(-\Delta)^s Q + Q - |Q|^{\alpha} Q = 0 \ \ \mbox{in} \ \ \R^N.
\ee
We refer to \cite{FrLe-10} and references given there for physical applications of this problem. Here and throughout the following, we assume that the exponent in the nonlinearity satisfies 
\be  \label{ineq:alpha}
0 < \alpha < \alphcr(s,N) ,
\ee
where we set
\be \label{def:alphacr}
 \quad \alphcr(s,N) := \left \{ \begin{array}{ll} \frac{4s}{N-2s} & \quad \mbox{for $0 < s < \frac{N}{2}$,} \\ +\infty & \quad \mbox{for $s \geq \frac{N}{2}$.} \end{array} \right . 
\ee
The condition that $\alpha$ be strictly less than $\alphcr(s,N)$ ensures that the nonlinearity in \eqref{eq:Q} is {\em $H^s$-subcritical}. Indeed, by Pohozaev-type identites (see also below), it can be shown that equation \eqref{eq:Q} does not admit any nontrivial solutions in $(H^s \cap L^{\alpha+2})(\R^N)$ when $\alpha \geq \alphcr$ holds. Thus the condition \eqref{ineq:alpha} is necessary for the existence of nontrivial solutions of \eqref{eq:Q}, but it is also sufficient as we now recall.

A natural approach to construct nonnegative nontrivial solutions for equation \eqref{eq:Q} is given by considering the fractional Gagliardo--Nirenberg--Sobolev (GNS) inequality
\be \label{ineq:GN}
\int_{\R^N} | u |^{\alpha+2} \leq C_{\mathrm{opt}} \left ( \int_{\R^N} | (-\DD)^{s/2} u |^2 \right )^{\frac{N \alpha}{4s}} \left ( \int_{\R^N} |u|^2 \right )^{\frac{\alpha+2}{2} - \frac{N \alpha}{4s}}.
\ee
Here $C_{\mathrm{opt}} > 0$ denotes the sharp constant (depending on $s,N,\alpha$) which can be obtained by minimizing the corresponding ``Weinstein functional'' (see \cite{We-87}) given by
\be
J(u) = \frac{ \left (  \int | (-\DD)^{s/2} u |^2 \right )^{\frac{N \alpha}{4s}} \left ( \int |u|^2 \right )^{\frac{\alpha}{4s}(2s-N) + 1} }{\int |u|^{\alpha+2} }
\ee
defined for $u \in H^s(\R^N)$ with $u \not \equiv 0$. Obviously, any minimizer $Q \in H^s(\R^N)$ for $J(u)$ optimizes the interpolation estimate \eqref{ineq:GN} and vice versa. By methods of variational calculus (see below), we find that $C_{\mathrm{opt}}^{-1} = \inf_{u \not \equiv 0} J(u) >0$ is indeed attained. Moreover, any minimizer $Q \in H^s(\R^N)$ for $J(u)$ is easily found to satisfy equation \eqref{eq:Q} after a suitable rescaling $Q \mapsto \mu Q(\lambda \cdot)$ with some constants $\mu$ and $\lambda$. Since $J(|u|) \leq J(u)$ holds, we can also deduce that minimizers $Q \geq 0$ for $J(u)$ can be chosen to be nonnegative. 

We summarize the following existence result along with fundamental properties of nonnegative solutions for equation \eqref{eq:Q}.
\begin{prop} \label{prop:Q}
Let $N \geq 1$, $s \in (0,1)$, and $0 < \alpha < \alphcr(s,N)$. Then the following holds.
\begin{enumerate}
\item[(i)] {\bf Existence:} There exists a minimizer $Q \in H^s(\R^N)$ for $J(u)$, which can be chosen a nonnegative function $Q\geq0$ that solves equation \eqref{eq:Q}.

\item[(ii)] {\bf Symmetry, regularity, and decay:} If $Q \in H^s(\R^N)$ with $Q \geq 0$ and $Q \not \equiv 0$ solves \eqref{eq:Q}, then there exists some $x_0 \in \R^N$ such that $Q(\cdot -x_0)$ is radial, positive and strictly decreasing in $|x-x_0|$. Moreover, the function $Q$ belongs to $(H^{2s+1} \cap C^\infty)(\R^N)$ and it satisfies
$$
\frac{C_1}{1+|x|^{N+2s}} \leq Q(x) \leq \frac{C_2}{1+|x|^{N+2s}} \ \ \mbox{for} \ \ x \in \R^N,
$$
with some constants $C_2 \geq C_1 > 0$ depending on $s,N, \alpha$, and $Q$.
\end{enumerate}
\end{prop}

\begin{proof}
These assertions follows from results in the literature. For instance, part (i) can be inferred by following \cite{We-87, AlBoSa-97} where the existence of minimizers for $J(u)$ for $N=1$ is shown by concentration-compactness arguments; the generalization to $N\geq 2$ is straightforward. As an alternative, we provide a simple existence proof without concentration-compactness arguments, by using rearrangement inequalities; see Appendix \ref{app:Q}.

As for the symmetry result in (ii), we can apply the moving plane method in \cite{MaZh-10} for nonlocal equations. See Appendix \ref{app:Q} again, where we also give some details regarding the assertions about decay and regularity of $Q$. Note that the properties stated in (ii) follow if $Q \in H^s(\R^N)$, $Q \not \equiv 0$, is only assumed to be a nonnegative solution of equation \eqref{eq:Q}, but in particular this applies to the minimizing solution given in (i).   \end{proof}

To formulate our main results about nonnegative solutions of equation \eqref{eq:Q}, we introduce the following notion of ground state solutions.

\begin{definition} \label{def:Qground}
Assume that $Q \in H^s(\R^N)$ is a real-valued solution of equation \eqref{eq:Q}. Let $L_+$ denote the corresponding linearized operator given by
$$
L_+ = (-\Delta)^s + 1 - (\alpha+1) |Q|^{\alpha} 
$$
acting on $L^2(\R^N)$. We say that $Q \geq 0$ with $Q \not \equiv 0$ is a {\bf ground state solution} of equation \eqref{eq:Q}, if $L_+$ has Morse index equal to one, i.\,e., $L_+$ has exactly one strictly negative eigenvalue (counting multiplicity).
\end{definition}

\begin{remarks*} {\em
1.) From \eqref{eq:Q} itself it directly follows that $(Q, L_+ Q) = - \alpha \int |Q|^{\alpha+2} < 0$. Hence, by the min-max principle, the operator $L_+$ has at least one negative eigenvalue for any nontrivial real-valued solution $Q \in H^s(\R^N)$.

2.) If $Q \geq 0$ is a (local) minimizer of the Weinstein functional $J(u)$, it is straightforward to see that $L_+$ has Morse index equal to one; see the proof of Corollary \ref{cor:GN} below. In particular, if $Q$ optimizes \eqref{ineq:GN} then $Q$ is a ground state in the sense of Definition \ref{def:Qground}. 

3.) Note that the notion of ground states defined above is weaker than the one used in \cite{FrLe-10}, where $Q$ was assumed to be a global minimizers of $J(u)$. 
}
\end{remarks*} 

The following result about ground state solutions for \eqref{eq:Q} establishes the key fact that the corresponding linearized operator is nondegenerate.  

\begin{thm}{\bf (Nondegeneracy).} \label{thm:nondeg}
Let $N \geq 1$, $s \in (0,1)$, and $0 < \alpha < \alphcr(s,N)$. Suppose that $Q \in H^s(\R^N)$ is a ground state solution of \eqref{eq:Q}. Then the linearized operator $L_+$ is nondegenerate, i.\,e., its kernel is given by
$$
\mathrm{ker} \, L_+ = \mathrm{span} \, \big \{ \partial_{x_1} Q, \ldots, \partial_{x_n} Q \big \} .
$$ 
\end{thm}

\begin{remarks*} {\em
1.) Suppose $Q = Q(|x|)$ is a radial ground state (which by Proposition \ref{prop:Q} follows after a translation). Then Theorem \ref{thm:nondeg} implies that $(\mathrm{ker} \, L_+ ) \cap L_{\mathrm{rad}}^2(\R^N) = \{ 0 \}$ and we easily check that $L_+$ is invertible on $L^2_{\mathrm{rad}}(\R^N)$. 

2.) The nondegeneracy of $L_+$ implies the coercivity estimate
$$
(u, L_+ u) \geq c \| u \|_{H^s}^2 \ \ \mbox{for} \ \ u \perp M,
$$   
with some positive constant $c > 0$, where $M$ is a suitably chosen $(n+1)$-dimensional subspace (e.\,g., one can take $M = \mathrm{span} \, \{ \phi, \partial_{x_1} Q, \ldots \partial_{x_n} Q \}$ with $\phi$ being the linear ground state of $L_+$.) Such results form a key aspect in the stability and blowup analysis for related time-dependent problems (e.\,g.,~generalized Benjamin--Ono equations, fractional Schr\"odinger equations etc.); see, e.\,g., \cite{KeMaRo-11, KrLeRa-12} for applications. 
}
\end{remarks*}

Finally, we have the following uniqueness result for ground state solutions of \eqref{eq:Q}, which generalizes the result in \cite{FrLe-10} to arbitrary space dimensions.

\begin{thm} {\bf (Uniqueness).} \label{thm:unique}
Let $N \geq 1$, $s \in (0,1)$, and $0 < \alpha < \alpha_*(s,N)$. Then the ground state solution $Q \in H^s(\R^N)$ for equation \eqref{eq:Q} is unique up to translation.  
\end{thm}

As a consequence of this uniqueness result, we have the following classification of the optimizers for inequality \eqref{ineq:GN}.

\begin{cor} \label{cor:GN}
Every optimizer $v \in H^s(\R^N)$ for the Gagliardo--Nirenberg--Sobolev inequality \eqref{ineq:GN} is of the form $v = \beta Q(\gamma ( \cdot + y))$ with some $\beta \in \C$, $\beta \neq 0$, $\gamma > 0$, and $y \in \R^N$.
\end{cor}

\begin{proof}[Proof of Corollary \ref{cor:GN}]
With Theorem \ref{thm:unique} at hand, we can follow the arguments for corresponding result in $N=1$ dimension given in \cite{FrLe-10}. That is, by strict rearrangement inequalities for $(-\DD)^s$ with $s\in (0,1)$ (see \cite{BuHa-06,FrSe-08}) we deduce that any optimizer $v \in H^s(\R^N)$ for \eqref{ineq:GN} is of the form $v = \beta v^*(\cdot + y)$ for some $\beta \in \C$, $\beta \neq 0$ and $y \in \R^N$, where $v^*=v^*(|x|) \geq 0$ denotes the symmetric-decreasing rearrangement of $v$. Since $v^* \in H^s$ is also a minimizer for $J(u)$, we see (after rescaling $v^* \mapsto \lambda v^*(\mu \cdot)$ if necessary) that $v^* = Q \geq 0$ satisfies equation \eqref{eq:Q}. Furthermore, an explicit calculation using the positivity of the second variation $\frac{d^2}{d \eps^2} J(Q+\eps \eta) \big |_{\eps = 0} \geq 0$ for any $\eta \in C^\infty_0(\R^N)$ shows that $L_+$ has Morse index equal to one. (This argument to determine the Morse index of $L_+$ for a minimizer $Q$ was introduced by M.~Weinstein \cite{We-85} in the context of NLS. See \cite{FrLe-10} for details of its adaption for fractional Laplacians in $N=1$ dimension; the generalization to $N \geq 2$ is immediate.) Hence $v^*=Q \geq 0$ is a ground state for equation \eqref{eq:Q} and  we can apply Theorem \ref{thm:unique} to conclude the proof. \end{proof}

\section{Proof of Theorem 1}

\label{sec:mono}

\subsection{Preliminaries}

We start by briefly recalling the extension principle in \cite{CaSi-07} that expresses the nonlocal operator $(-\DD)^s$ on $\R^N$ with $s \in (0,1)$ as a Dirichlet--Neumann map for a suitable local elliptic problem posed on the upper halfspace $\R^{N+1}_+$. See also also \cite{GrZw-03, ChGo-11, MoOs-69}, where this observation appears in the contexts of conformal geometry and stochastic processes, respectively. 

Let $s \in (0,1)$ be given. For a measurable function $f : \R^N \to \R$, we define its {\em $s$-Poisson extension} to the upper halfspace $\R^{N+1}_+$ by setting
\be
(\cE_{s} f)(x,t) = \int_{\R^N} P_s(x-y,t) f(y) \, dy.
\ee 
Here the generalized Poisson kernel $P_s(z,t)$ of order $s$ is given by
\be
P_s(z,t) = \frac{1}{t^N} k_s \left ( \frac x t \right ), \quad  \mbox{where} \ \ k_s(z) = \frac{c_{N,s}}{(1 + |z|^2)^{\frac{N+2s}{2}}} ,
\ee
where the constant $c_{N,s} > 0$ is chosen such that $\int_{\R^N} k_s \, dz= 1$ holds.  Under suitable assumptions on $f$ (see, e.\,g.,\cite{CaSi-07,CaSi-10}), it is known that $w(x,t)= (\cE_sf)(x,t)$ solves  the degenerate elliptic boundary-value problem
\be \label{eq:extension-s}
\left \{ \begin{array}{ll} \mathrm{div} \, (t^{1-2s} \nabla w) = 0 & \quad \mbox{in $\R^{N+1}_+$} , \\
w = f &  \quad \mbox{on $\partial \R^{N+1}_+$.} \quad \end{array} \right .
\ee
Here the boundary condition is understood in some suitable sense of traces; see also Sect.~\ref{sec:osc} below. If $f$ is sufficiently regular, then we have (in some approriate space) the convergence
\be \label{eq:DirichNeu}
- d_s \lim_{t \to 0^+} t^{1-2s} \partial_t w(\cdot, t) = (-\DD)^s f,
\ee
where $d_s > 0$ is the constant in \eqref{eq:ds}. Note that \eqref{eq:DirichNeu} expresses the fact that $(-\DD)^s$ can be regarded as the  Dirichlet--Neumann map for problem \eqref{eq:extension-s} with the weight $t^{1-2s}$. 

In the special case $s=1/2$, the above observations reduce to the classical fact that, if $f : \R^N \to \R$ is continuous and bounded, then the Poisson extension $w = \cE_{1/2}f$ is the unique bounded harmonic function in $\R^{N+1}_+$ continuous up to the boundary such that $w(x,0) = f(x)$. In fact, this result carries over to the whole range $s \in (0,1)$, by results for the degenerate elliptic operator $L_s = \mathrm{div} ( t^{1-2s} \nabla \cdot)$ derived by Fabes et al.~in \cite{FaKeSe-82}; see also \cite{CaSi-07,CaSi-10}.

\subsection{Monotonicity Formula}
Let $u \in L^\infty(\R^N)$ satisfy the assumptions of Theorem \ref{thm:linear1}. By Proposition \ref{prop:regu}, we have the regularity estimate $u \in C^{1,\beta}(\R^N)$ for some $\beta \in (0,1)$.

Next, we introduce the following convenient slight abuse of notation: Let $u=u(x,t)$ with $(x,t) \in \R^{N+1}_+$ denote the $s$-Poisson extension of $u=u(|x|)$ to the upper halfspace $\R^{N+1}_+$. Since $u(|x|)$ is radial on $\R^{N}$, we clearly have that its corresponding extension $u=u(|x|, t)$ is cylindrically symmetric on $\RNup$ with respect to the $t$-axis. Using this fact, we can write the boundary problem \eqref{eq:extension-s} satisfied by the extension $u$ as follows:
\begin{equation} \label{eq:ext}
\left \{ \begin{array}{ll} \displaystyle u_{rr} + \frac{N-1}{r} u_r +  u_{tt} + \frac{a}{t} u_t = 0, & \quad \mbox{in $\R^{N+1}_+$}, \\[1ex]
\displaystyle - d_s  t^{a} u_t + V u = 0, & \quad \mbox{on $\partial \R^{N+1}_+$},
\end{array} \right . 
\end{equation}
with the constant $d_s>0$ taken from \eqref{eq:ds}. Here and in the following, we use convention that $a=1-2s$ for $s \in (0,1)$ given.

Inspired by the work of Cabr\'e and Sire \cite{CaSi-10} (see also \cite{CaSo-05} for earlier work in the case $s=1/2$) on layer and radial solutions of nonlinear equations of the form $(-\DD)^s v = f(v)$ on $\R^N$, we introduce the function 
\begin{equation} \label{def:Hfirst}
H(r) = d_s   \int_0^{+\infty} \frac{t^a}{2} \left \{ u_r^2(r,t) - u_t^2(r,t) \right \} dt - \frac{1}{2} V(r) u(r)^2 ,
\end{equation}
From the estimates in Proposition \ref{prop:regu2} for the extension $u=u(r,t)$, we deduce that $H(r)$ is a well-defined and continuous function. Moreover, we see that
\be \label{eq:Hinf}
\lim_{r \to +\infty} H(r) = 0,
\ee
\begin{equation} \label{eq:H0}
H(0) = - d_s \int_0^{+\infty } \frac{t^a}{2} u_t^2(0,t) \, dt - \frac{1}{2} V(0) u(0)^2 \leq - \frac 1 2 V(0) u(0)^2 .
\end{equation} 
Note that \eqref{eq:Hinf}  follows from $\lim_{r \to +\infty} \int_0^{+\infty} t^a \{ u_r^2 - u_t^2 \} (r,t) \, dt = 0$ by Proposition \ref{prop:regu2} and the fact that $\lim_{r \to +\infty} V(r) u(r)^2 = 0$, since $V \in L^\infty$ and $u(|x|) \to 0$ as $|x| \to \infty$; see the remark following Theorem \ref{thm:linear1}. To conclude \eqref{eq:H0}, we just use the fact that $u_r(0,t) \equiv 0$ holds by cylindrical symmetry. 

Let us first sketch the argument to prove Theorem \ref{thm:linear1} by a formal calculation. Recall that $u(r)$ is a $C^1$ function. Furthermore, for the moment let us also assume that  $V$ is differentiable too (and not just weakly differentiable with $V' \in L^1_{\mathrm{loc}}$). Assuming that we are allowed to differentiate under the integral sign in \eqref{def:Hfirst}, we (formally at least) obtain by using equation \eqref{eq:ext} and integrating by parts (for details see below) that
\be \label{ineq:H}
\frac{dH}{dr} = -d_s \frac{N-1}{r} \int_0^{+\infty} t^a u_r^2(r,t) \, dt - \frac 1 2 V'(r)u(r)^2 \leq 0,
\ee
since $V' \geq 0$ by assumption. Hence $H(r)$ is monotone decreasing and we conclude that
\be
-\frac 1 2 V(0) u(0)^2 \geq H(0) \geq H(r) \geq \lim_{r\to +\infty} H(r) = 0. 
\ee
Suppose now that $u(0)=0$. Then equality holds in the above inequalities and therefore $H(r) \equiv 0$ and consequently $dH/dr \equiv 0$. Now let us assume that $N \geq 2$ holds. Then, we conclude from \eqref{ineq:H} and $V' \geq 0$ that $u_r(r,t) \equiv 0$ holds. Hence $u \equiv 0$ follows for $N \geq 2$. (The proof for $N=1$ is actually a bit more involved $N =1$; see below). This completes the proof of Theorem \ref{thm:linear1} for $N \geq 2$, provided that we can differentiate under the integral sign in the expression for $H(r)$. However, this is not guaranteed in general for $u \in C^{1,\beta}(\overline{\RNup})$, as one can check by inspection. To handle this technicality, we could impose more regularity on $V$ to guarantee that $u \in C^{2,\beta}(\overline{\RNup})$ holds (which would be sufficient to justify interchanging differentiation and integration). However, we will keep the weaker regularity conditions on $V$, by using a regularized version of the previous arguments as follows.

Let $\eta \in C^\infty_0( \R_+ )$ with $0 \leq \eta \leq 1$ be a nonnegative bump function with $\int_0^{+\infty} \eta(u) \, du = 1$. We define an averaged version of $H(r)$ given by
\be
\Hav(r) =  \int_0^{+\infty} H(\bar{r}) \eta \left ( \frac{\bar{r}}{r} \right ) \frac{d \bar{r} }{r}
\ee
Clearly, the function $\Hav(r)$ is differentiable and taking the derivative with respect to $r$ interchanges with integration. Furthermore, by using change of variables, dominated convergence and the normalization condition $\int_0^{+\infty} \eta(u) \, du = 1$, we readily check that
\be  \label{eq:Havlim}
\lim_{r \to 0^+} \Hav(r) = H(0)  \ \ \mbox{and} \ \  \lim_{r \to +\infty} \Hav(r) =  0,
\ee
recalling that \eqref{eq:Hinf} holds. Next, we claim the following fact.

\begin{lemma} \label{lem:Hav}
It holds that
$$
\Hav'(r) = -\int_{\bar{r}=0}^{+\infty} \left \{ d_s \frac{N-1}{\bar{r}} \int_{t=0}^{+\infty} t^a u_r^2(\bar{r},t) \, dt + \frac{1}{2} V'(\bar{r}) u(\bar{r})^2 \right \} \eta \left ( \frac{\bar{r}}{r} \right ) \frac{d \bar{r}}{r} .
$$
In particular, we have $\Hav'(r) \leq 0$ and hence $\Hav(r)$ is monotone decreasing.
\end{lemma}

\begin{remark*} {\em
Note that $V'(r)$ denotes the weak derivative of $V$. Since $V \in C^{0,\gamma}(\R^N)$, we have $V' \in L^{1}_{\mathrm{loc}}(\R^N)$. 
}

\end{remark*}

\begin{proof}
First, we note that
\begin{align*}
\Hav'(r) & =  \int_0^{+\infty} H(\bar{r}) \partial_r \left ( \eta \left ( \frac{\bar{r}}{r} \right ) \right ) \frac{d\bar{r}}{r} - \frac{1}{r} \Hav(r) \\
& = - \int_{\bar{r}=0}^{+\infty} \left \{ d_s \int_{t=0}^{+\infty} \frac{t^a}{2} \left \{ u_r^2 - u_t^2 \right \}(\bar{r},t) \, dt - \frac 1 2 V(\bar{r}) u(\bar{r})^2 \right \} \partial_{\bar{r}} \left \{ \eta \left ( \frac{\bar{r}}{r} \right ) \right \} \frac{\bar{r}}{r^2} \, d\bar{r} - \frac{1}{r} \Hav(r),
\end{align*}
since  $\partial_r \{ \eta ( \frac{\bar{r}}{r}) \} = - \partial_{\bar{r}} \{ \eta( \frac{\bar{r}}{r} ) \}  \frac{\bar{r}}{r} $. Next, by applying Fubini's theorem and integrating by parts with respect to $\bar{r}$ and applying Fubini's theorem again, we obtain
\begin{align*}
\Hav'(r) & =  \int_{\bar{r}=0}^{+\infty} \left \{  d_s \int_{t=0}^{+\infty} t^a \left \{ u_r u_{rr} - u_t u_{tr} \right \} (\bar{r},t) \, dt - V(\bar{r} ) u(\bar{r}) u_r(\bar{r}) \right . \\
& \qquad \left . - \frac{1}{2} V'(\bar{r}) u(\bar{r})^2  \right \} \eta \left ( \frac{\bar{r}}{r} \right ) \frac{d \bar{r} }{r} \\
& = - \int_{\bar{r}=0}^{+\infty} \left \{ d_s \frac{N-1}{\bar{r}} \int_{t=0}^{+\infty} u_r^2(\bar{r},t) \, dt + d_s \int_{t=0}^{+\infty} \partial_t \left \{ t^a u_t u_r Ê\right \} (\bar{r},t) \, dt + V(\bar{r}) u(\bar{r}) u_r(\bar{r}) \right . \\
& \left . \qquad + \frac{1}{2} V'(\bar{r}) u(\bar{r})^2 \right \} \eta \left ( \frac{\bar{r}}{r}  \right) \frac{d \bar{r}}{r} \\
& = - \int_{\bar{r}=0}^{+\infty} \left \{ d_s \frac{N-1}{\bar{r}} \int_{t=0}^{+\infty} u_r^2(\bar{r},t) \, dt + d_s t^a u_t u_r (\bar{r},t) \Big |_{t=0}^{+\infty} + V(\bar{r}) u(\bar{r}) u_r(\bar{r}) \right . \\
& \qquad \left . +\frac{1}{2} V'(\bar{r}) u(\bar{r}^2) \right \}  \eta \left ( \frac{\bar{r}}{r}  \right) \frac{d \bar{r}}{r} \\
& = - \int_{\bar{r}=0}^{+\infty} \left \{ d_s \frac{N-1}{\bar{r}} \int_{t=0}^{+\infty} t^a u_r^2(\bar{r},t) \, dt + \frac{1}{2} V'(\bar{r}) u(\bar{r})^2 \right \} \eta \left ( \frac{\bar{r}}{r} \right ) \frac{d \bar{r}}{r} ,
\end{align*}
which is the desired formula. Notice that, in the first two steps, we used equation \eqref{eq:ext}. Also, note that $\lim_{t \to +\infty} t^a u_t u_r = 0$ due to the decay estimates in Proposition \ref{prop:regu2}.
\end{proof}

\subsection{Completing the Proof of Theorem \ref{thm:linear1}}
Assume that $u(0) = 0$ holds. By \eqref{eq:Havlim} and \eqref{eq:H0}, this implies that $\Hav(0) \leq 0$. On the other hand, we have $\lim_{r \to + \infty}\Hav(r) = 0$ by \eqref{eq:Havlim}. Because $\Hav(r)$ is monotone decreasing thanks to Lemma \ref{lem:Hav}, we conclude 
\be
\Hav(r) \equiv 0 \ \ \mbox{and} \ \ \Hav'(r) \equiv  0.
\ee
 We discuss the cases $N \geq 2$ and $N=1$ separately as follows.

\subsubsection*{Case $N \geq 2$}  Using Lemma \ref{lem:Hav} and the assumption $V'(r) \geq 0$ for a.\,e.~$r$, we deduce that
\be
\int_{\bar{r}=0}^{+\infty} \left \{ \int_{t=0}^{+\infty} u_r^2(\bar{r},t) \, dt  \right \} \eta \left ( \frac{\bar{r}}{r} \right ) \frac{d \bar{r}}{r} = 0 \quad \mbox{for all $r > 0$}.
\ee
Since this holds for any $\eta(\cdot) \in C_0^\infty(\R_+)$ with $0 \leq \eta \leq 1$ with $\int_0^\infty \eta(u) \, du=1$, we conclude that $\int_{t=0}^{+\infty} u_r^2(r,t) \, dt=0$ for almost every $r$. By continuity of $u(r,t)$, this shows that $u_r(r,t) \equiv 0$  and therefore $u(r,t)$ only depends on $t$. But this implies $u(r) = \mbox{const}.$  and hence $u(r) \equiv 0$, because $u(r) \to 0$ as $r \to +\infty$. This completes the proof of Theorem \ref{thm:linear1} for any dimension $N \geq 2$.

\subsubsection*{Case $N=1$} In this case, we have to provide some additional arguments, since the integral term in Lemma \ref{lem:Hav} containing $u_r^2(r,t)$ is absent when $N=1$. In fact, by assuming that $V' > 0$ for a.\,e.~$r$, we could easily derive from Lemma \ref{lem:Hav} that $u\equiv 0$ holds, using the potential term in that identity. However, we shall now give a proof that only assumes that $V' \geq 0$ holds for a.\,e.~$r$.

Indeed, since we have $\Hav(r) \equiv 0$ (for any bump function $\eta$ as above), we deduce that $H(r) \equiv 0$. From  \eqref{eq:H0} and the assumption that $u(0)=0$, we deduce that
$$
\int_0^{+\infty}  u_t^2(0,t) t^{1-2s} \,dt = 0 .
$$
By continuity of $u(r,t)$, this shows that $u_t(0,t) = 0$ for every $t > 0$. We now prove that this implies that $u\equiv 0$ as follows. Recall that $u(r,t)$ is given by the $s$-Poisson extension. Thus, for every $t> 0$, we have
\be \label{eq:poisson}
u(0,t) = c_{n,s} \int_0^{+\infty} \frac{t^{2s} u(y)}{(t^2 + y^2)^{(1+2s)/2}} \, dy 
\ee
with some constant $c_{n,s} > 0$. By assumption $u(y)$ is bounded and vanishes at infinity, which implies that $\lim_{t \to +\infty} u(0,t) = 0$ from \eqref{eq:poisson}. (Indeed, for every $\eps>0$ there is an $R_\eps>0$ such that $|u(y)| \leq M \chi_{\{|y|\leq R_\eps\}} + \eps \chi_{\{|y|> R_\eps\}}$. Now plug this into the integral above and use the fact the integral of $t^{2s} (t^2+y^2)^{-(1+2s)/2}$ is finite and independent of $t$.) Recalling that $u_t (0,t) = 0$ for every $t > 0$, we conclude that $u(0,t) = 0$ for every $t >0$. Thus, by repeated differentiation of \eqref{eq:poisson} with respect to $t>0$ and choosing $t=1$, we obtain that
\be
\int_0^{+\infty} \frac{u(y)}{(1+ y^2)^{(1+2s)/2+k}} \,dy = 0
\quad\text{for every}\ k\in\N_0 \,.
\ee
Now, we define a function $f$ on $[0,1]$ by $f(1/(1+y^2)) = (2y)^{-1} (1+y^2)^{-(-1+2s)/2} u(y)$ and change variables to $\alpha=1/(1+y^2)$. This gives us
\be \label{eq:weier}
\int_0^1 \alpha^k f(\alpha) \,d\alpha = 0
\quad\text{for every}\ k\in\N_0 \,.
\ee
Since $\int_0^\infty |f(\alpha)|\,d\alpha = \int_0^\infty (1+y^2)^{-(1+2s)/2} |u(y)|\,dy$ is finite, we see that $f(\alpha)\,d\alpha$ is a finite signed measure on $[0,1]$. By Weierstrass' theorem and the Riesz representation theorem, we conclude from \eqref{eq:weier} that $f\equiv 0$ holds, which implies  that $ u\equiv 0$, as desired.

The proof of Theorem \ref{thm:linear1} is now complete. \hfill $\square$

\begin{remark*} {\em
The proof of Theorem \ref{thm:linear1} actually shows that $u \not \equiv 0$ if and only if $u(0) \neq 0$ and $V(0) < 0$.}
\end{remark*}

\section{Nodal Bounds via Extension to $\R^{N+1}_+$}

\label{sec:osc}

The present section serves as a preparation for the proof of  Theorem 2. We will derive oscillation bounds for radial eigenfunctions for fractional Schr\"odinger operators $H= (-\DD)^s + V$ on $\R^N$, where the potential $V$ is assumed to satisfy a mild condition (i.\,e., $V$ belongs to an appropriate Kato class.)  As in \cite{FrLe-10}, the strategy in the section is based on a related variational problem posed for functions on the upper half-space $\R^{N+1}_+$. This section follows the related arguments given in \cite{FrLe-10} for $N=1$. Therefore, the following discussion will be rather brief and without details except when necessary. 

However, a decisive difference to \cite{FrLe-10} will be that the oscillation bound for radial eigenfunctions derived in Proposition \ref{prop:osc} below will be not optimal in $N \geq 2$ dimensions. The reason for this is of topological nature stemming from the fact that the set $\R^{N+1}_+ \setminus \{ (0,t) : t > 0\}$ is connected for $N \geq 2$ in contrast to the case when $N=1$. By an additional strategy, we will improve the oscillation bound stated in Proposition \ref{prop:osc} in an optimal way, provided that the potential $V$ additionally satisfies the additional conditions (V1) and (V2). This will be carried out in Section \ref{sec:proofthm2} below, where the proof of Theorem \ref{thm:linear2} will be given.

\subsection{Variational Formulation on $\R^{N+1}_+$}
Let $N \geq 1$ and $s \in (0,1)$ be given. We consider a general class of fractional Schr\"odinger operators
\be
H=(-\DD)^s + V, 
\ee
where the potential $V : \R^N \to \R$ belongs to the so-called Kato class for $(-\DD)^s$ in $\R^N$. We shall denote this condition by $V \in K_s(\R^N)$ in what follows. From \cite{CaMaSi-90} we recall that a measurable function $V : \R^N \to \R$ belongs to $K_s(\R^N)$ if and only if 
\be
\lim_{E \to +\infty} \| ((-\DD)^s + E)^{-1} |V| \|_{L^\infty \to L^\infty} = 0.
\ee
For the readers less familiar with Kato classes, we list the following facts (see, e.\,g.~\cite{CaMaSi-90}).
\begin{itemize}
\item If $V \in K_s(\R^N)$, then $V$ is infinitesimally relatively bounded with respect to $(-\DD)^s$. Therefore $H=(-\DD)^s + V$ defines a unique self-adjoint operator on $L^2(\R^N)$ with form domain $H^s(\R^N)$, and the operator $H$ is bounded from below.
\item If $V \in L^{p}(\R^N)$ with some $\max \{ 1,\frac{N}{2s} \} < p \leq +\infty$, then $V \in K_s(\R^N)$. 
\item If $\psi \in L^2(\R^N)$ is an eigenfunction $H=(-\DD)^s + V$ with $V \in K_s(\R^N)$, then $\psi$ is bounded and continuous.
\end{itemize}
Since $H=(-\DD)^s +V$ is real operator (mapping real functions to real functions), its eigenfunctions can be chosen real-valued, which we will assume from now on. 

Following \cite{FrLe-10}, we now seek a variational characterization of the eigenvalues of $H=(-\DD)^s + V$ in terms of a {\em local energy functional} by using the extension to the upper half-space $\R^{N+1}_+$. From the previous section, we recall the definition 
\be
a=1-2s
\ee 
for $s \in (0,1)$ given. We introduce the functional
\be \label{def:hfrak}
\mathfrak{H}(u) = d_s \iint_{\R^{N+1}_+} |\nabla u|^2 t^a \, dx \, dt + \int_{\R^N} V(x) |u(x,0)|^2 \, dx 
\ee
defined for $u \in \mathcal{H}^{1,a}(\R^{N+1}_+)$, where $u(x,0)$ denotes its trace on $\partial \R^{N+1}_+$ (see below). As usual $d_s > 0$ denotes the constant from \eqref{eq:ds}. The space $\mathcal{H}^{1,a}(\R^{N+1}_+)$ is given by
\be
\mathcal{H}^{1,a}(\R^{N+1}_+) = \{ u \in \dot{\mathcal{H}}^{1,a}(\R^{N+1}_+) : u(x,0) \in L^2(\R^N)  \}
\ee
Here the space $\dot{\mathcal{H}}^{1,a}(\R^{N+1}_+)$ is defined as the completion of $C^\infty_0(\R^{N+1}_+)$ with respect to the homogeneous Sobolev norm
\be
\| u \|_{\dot{\mathcal{H}}^{1,a}}^2 = \iint_{\R^{N+1}_+} |\nabla u|^2 t^{a} \, dx \, dt. 
\ee
By Hardy's inequality, we see that $\dot{\mathcal{H}}^{1,a}(\R^{N+1}_+)$ is a space of functions if $0 < s < N/2$ (note for $N \geq 2$ this holds true for all $s\in (0,1)$), whereas for $1/2 \leq s  <1$ when $N=1$ it is a space of functions modulo additive constants. (See also \cite{FrLe-10} for more details on this.)  By adapting the arguments in \cite{FrLe-10}, it can be seen that there exists a well-defined trace operator $T : \dot{\mathcal{H}}^{1,a}(\R^{N+1}_+) \to \dot{H}^s(\R^N)$, where we often write $(Tu)(x) = u(x,0)$ for notational convenience. Moreover, we have the sharp trace inequality
\be \label{ineq:trace1}
\iint_{\R^{N+1}_+} | \nabla u|^2 t^a \, dx \, dt \geq \frac{1}{d_s} \int_{\R^N} |(-\DD)^{s/2} Tu|^2 \, dx.
\ee 
As an amusing aside, we remark that the constant on the right-hand side $1/d_s$ does not depend on the dimension $N$.
Finally, we mention the following fact:
\be \label{ineq:trace3}
\mbox{Equality  holds in \eqref{ineq:trace1} if and only if $u =\cE_s f$ for some $f \in \dot{H}^s(\R^N)$.} 
\ee
Recall that $\cE_s f = P_s(t, \cdot) \ast f$ denotes the $s$-Poisson extension of $f : \R^N \to \R$ to the upper halfspace $\R^{N+1}_+$. Regarding the proofs of \eqref{ineq:trace1}--\eqref{ineq:trace3}, we remark that these assertions follow by an adaptation of the discussion in \cite{FrLe-10}. We omit the details. 

Since we are ultimately interested in $H=(-\DD)^s + V$ with radial potentials $V \in K_s(\R^N)$, it is natural to introduce the closed subspace 
\be
\mathcal{H}^{1,a}_{\mathrm{rad}}(\R^{N+1}_+) = \big \{ u \in \mathcal{H}^{1,a}(\R^{N+1}_+) : \mbox{$x \mapsto u(x,t)$ is radial in $x \in \R^N$ for a.\,e.~$t > 0$} \big \} .
\ee
Now we are ready for the following characterization of discrete eigenvalues of $H=(-\DD)^s + V$ in terms of the local energy functional $\mathfrak{H}(u)$ introduced above. See also \cite{BaKu-04} for a similar result for $\sqrt{-\Delta}$ on the interval $(-1,1)$.

\begin{prop} \label{prop:varia}
Let $N \geq 1$, $0 < s < 1$, and $V \in K_s(\R^N)$. Suppose that $n \geq 1$ is an integer and assume that $H=(-\DD)^s + V$ acting on $L^2(\R^N)$ has at least $n$ eigenvalues
$$
E_1 \leq E_2 \leq \cdots \leq E_n < \inf \sigma_{\mathrm{ess}}(H) .
$$
Furthermore, let $M$ be an $(n-1)$-dimensional subspace of $L^2(\R^N)$ spanned by eigenfunctions corresponding to the eigenvalues $E_1, \ldots, E_{n-1}$. Then we have
$$
E_n = \inf \Big \{ \mathfrak{H}(u) : u \in \mathcal{H}^{1,a}(\R^{N+1}_+), \;  \int_{\R^N} |u(x,0)|^2 \, dx = 1, \,  u(\cdot,0) \perp M Ê\Big \}
$$
with $\mathfrak{H}(u)$ defined in \eqref{def:hfrak}. Moreover, the infimum is attained if and only if $u=\cE_a f$ with $f \in H^s(\R^N)$, where $\| f\|_2^2 =1$ and $f\in M^\bot$ is a linear combination of eigenfunctions of $H$ corresponding to the eigenvalue $E_n$.

Finally, if $V \in K_s(\R^N)$ is radial, then the same result holds true with $L^2(\R^N)$ and $\mathcal{H}^{1,a}(\R^{N+1}_+)$ replaced by $L^2_{\mathrm{rad}}(\R^N)$ and $\mathcal{H}^{1,a}_{\mathrm{rad}}(\R^{N+1}_+)$, respectively, and the eigenvalues counted in the corresponding subspaces.
\end{prop}

\begin{proof}
We argue in the same way as in \cite{FrLe-10}. That is, by \eqref{ineq:trace1}, we see that the infimum on the right-hand side is bounded from below by
$$
\inf \Big \{ \int_{\R^N} | (-\DD)^{s/2} f|^2 \,dx + \int_{\R^N} V|f|^2 \, dx : f \in H^s(\R^N), \; \| f \|_{L^2} = 1, \; f \perp M \Big \} ,
$$
where equality is attained if and only if $u = \cE_a f$, as we conclude from \eqref{ineq:trace3}. The assertions now follow from the usual variational characterization for the eigenvalues of $H$. 

Finally, suppose that $V \in K_s(\R^N)$ is radial and let $H=(-\DD)^s + V$ act on $L^2_{\mathrm{rad}}(\R^N)$. Now, we just note that if $f\in (L^2_{\mathrm{rad}} \cap H^{s})(\R^N)$ then $\cE_a f\in\mathcal{H}^{1,a}_{\mathrm{rad}}(\R^{N+1}_+)$, which follows from the fact that the convolution kernel $P_s(x,t)$ is a radial function of $x \in \R^N$. \end{proof}

With Proposition \ref{prop:varia} at hand, we now proceed by deriving bounds on the number of nodal domains for extension of eigenfunctions of $H$ to the upper half-space $\R^{N+1}_+$. Recall that $H=(-\DD)^s + V$ is a real operator, and hence any eigenfunction can be chosen real-valued. Furthermore, we recall that any eigenfunction $\psi$ of $H=(-\DD)^s+V$ with $V \in K_s(\R^N)$ is continuous. Therefore, its extension $\cE_a \psi$ belongs to $C^0(\overline{\R^{N+1}_+})$ and we can consider its nodal domains, which are defined as the connected components of the open set $\{ (x,t) \in \R^{N+1}_+ : (\cE_a \psi)(x,t) \neq 0 \}$ in the upper half-space $\R^{N+1}_+$. We have the following result based on \cite{FrLe-10}. See also \cite{BaKu-04, AlMa-94} for related results for $\sqrt{-\Delta}$ on an interval.

\begin{prop} \label{prop:nodal}
Let $N \geq 1$, $0 < s < 1$, and $V \in K_s(\R^N)$. Suppose that $n \geq 1$ is an integer and assume that $H=(-\DD)^s + V$ acting on $L^2(\R^N)$ has at least $n$ eigenvalues
$$
E_1 \leq E_2 \leq \cdots \leq E_n < \inf \sigma_{\mathrm{ess}}(H) .
$$
If $\psi_n \in H^s(\R^N)$ is a real eigenfunction of $H$ with eigenvalue $E_n$, then its extension $\cE_a \psi_n$, with $a=1-2s$, has at most $n$ nodal domains on $\R^{N+1}_+$.

Moreover, if $V \in K_s(\R^N)$ is radial, then the same result holds true with $L^2(\R^N)$ replaced by $L^2_{\mathrm{rad}}(\R^N)$ and $\psi_n \in H^s(\R^N)$ being a radial and real eigenfunction of $H$ for the $n$-th radial eigenvalue $\Er_n$.
\end{prop}

\begin{remark*} {\em
Note that the assertion about the case with radial potentials is an improvement in general: Suppose that  $V \in K_s(\R^N)$ is radial. Let $E_1 \leq E_2 \leq \ldots$ and let $E_{1}^{\mathrm{rad}} \leq E_2^{\mathrm{rad}} \leq \ldots$ denote the discrete eigenvalues of $H=(-\DD)^s + V$ acting on $L^2(\R^N)$ and acting on $L^2_{\mathrm{rad}}(\R^N)$, respectively. Then $E_n = E_m^{\mathrm{rad}}$ for some $m \leq n$.}
\end{remark*}

\begin{proof}
This follows from a variational argument in the spirit of Courant's nodal domain theorem. We follow the arguments given in \cite{FrLe-10}. For the reader's convenience, we provide the details of the proof as follows.

Suppose that $\cE_a \psi_n$ has nodal domains $\Omega_1, \ldots, \Omega_m \subset \R^{N+1}_+$ with $m \geq n+1$. Since $\cE_a \psi_n$ is continuous up to the boundary and $\psi_n \not \equiv 0$, we see that $\overline{\Omega}_j \cap \partial \R^{N+1}_+ \neq \emptyset$ for some $j=1,\ldots,m$. Without loss of generality, we assume that $j=1$ holds. Now, we consider the trial function
\be
u = \sum_{j=1}^n (\cE_a \psi_n) \gamma_j \mathds{1}_{\Omega_j},
\ee
where $\gamma_j \in \R$ are constants and $1_{A}$ denotes the characteristic function of a set $A \subset \R^{N+1}_+$. Note that $u \in \mathcal{H}^{1,a}(\R^{N+1}_+)$ with $\nabla u = \sum_{j=1}^n (\nabla \cE_a \psi_n) \gamma_j \mathds{1}_{\Omega_j}$. Next, let $M$ be an $(n-1)$-dimensional subspace of $L^2(\R^N)$ spanned by eigenfunctions of $H$ with eigenvalues $E_1, \ldots, E_{n-1}$. We can choose $\gamma_j \in \R$ such that $u(\cdot,0) \perp M$ and $\| u(\cdot,0) \|_{L^2} = 1$. Furthermore, following the arguments in \cite{FrLe-10}, a calculation yields that 
\be
\mathfrak{H}(u) = \lambda_n \sum_{j=1}^n |\gamma_j|^2 \int_{\overline{\Omega}_j \cap \partial \R^{N+1}_+} |u(x,0)|^2 \, dx = \lambda_n \| u(\cdot, 0) \|_{L^2}^2 = \lambda_n.
\ee
Thus equality is attained in Proposition \ref{prop:varia} and hence $u = \cE_s f$, where $f \in M^\perp$ is a linear combination of eigenfunctions of $H$ with eigenvalue $E_n$. In particular, the non-trivial function $u=\cE_a f$ satisfies 
\be
\mathrm{div} \, (t^a \nabla u) = 0 \ \ \mbox{in} \ \ \R^{N+1}_+.
\ee
Since $u \equiv 0$ in the open and non-empty set $\Omega_{n+1} \subset \R^{N+1}_+$, we conclude from unique continuation that $u \equiv 0$ on $\R^{N+1}_+$. But this is a contradiction. Hence $\cE_a \psi_n$ has at most $n$ nodal domains on $\R^{N+1}_+$.

Finally, we assume that $V \in K_s(\R^N)$ is radial and we consider $H=(-\DD)^s + V$ acting in $L^2_{\mathrm{rad}}(\R^N)$. In this case, the previous arguments apply in verbatim way by replacing $\mathcal{H}^{1,a}(\R^{N+1}_+)$ with $\mathcal{H}^{1,a}_{\mathrm{rad}}(\R^{N+1}_+)$.  Note that $\cE_a \psi_n \in \mathcal{H}^{1,a}_{\mathrm{rad}}(\R^{N+1}_+)$ whenever $\psi_n \in H^s(\R^N)$ is radial. Clearly, the nodal domains of $\cE_a \psi_n$ are cylindrically symmetric with respect to the $t$-axis. In particular, the trial function $u = \sum_{j=1}^n \ (\cE_a \psi_n) \gamma_j 1_{\Omega_j}$ belongs to $\mathcal{H}^{1,a}_{\mathrm{rad}}(\R^{N+1}_+)$.

The proof of Proposition \ref{prop:nodal} is now complete. \end{proof}

\subsection{An Oscillation Estimate}

First, we define the number of sign changes of a radial and continuous function $\psi$ on $\R^N$.

\begin{definition}
Let $\psi \in C^0(\R^N)$ be radial and let $M \geq 1$ be an integer. We say that $\psi(r)$ changes its sign $M$ times on $(0, +\infty)$, if there exist $0 < r_1 < \cdots < r_{M+1}$ such that $\psi(r_i) \neq 0$ for $i=1, \ldots, M+1$ and $\mathrm{sign} ( \psi(r_i) ) =- \mathrm{sign} ( \psi(r_{i+1}))$ for $i=1, \ldots,M$.
\end{definition}

We can now state the following oscillation estimate.

\begin{prop} \label{prop:osc}
Let $N \geq 1$ and $0 < s <1$. Suppose that $V \in K_s(\R^N)$ is a radial potential and consider $H=(-\DD)^s + V$ acting on $L^2_{\mathrm{rad}}(\R^N)$. Let $\Er_2 < \inf \sigma_{\mathrm{ess}}(H)$ be the second eigenvalue of $H$ acting on $L^2_{\mathrm{rad}}(\R^N)$ and suppose that $\psi_2 \in L^2(\R^N)$ is a radial and real-valued solution of $H \psi_2 = \Er_2 \psi_2$. Then $\psi_2$ changes its sign at most twice on $(0, +\infty)$.
\end{prop}

\begin{remark*}{\em 
 In Section \ref{sec:proofthm2} below, we will show that $\psi_2$ changes its sign exactly once on $(0,+\infty)$, provided the potential $V$ satisfies the additional conditions (V1) and (V2). Note that in $N=1$ dimension one can deduce a sharper bound as done in \cite{FrLe-10}.
}
\end{remark*}

\begin{proof} We follow the arguments in \cite{FrLe-10} by using nodal domain bounds for the extension of $\psi_2$ to the upper half-space $\R^{N+1}_+$.

We argue by contradiction. Suppose that $\psi_2(r)$ changes its sign at least three times on $(0,+\infty)$.   Thus there exist $0 < r_1 < r_2 < r_3 < r_4$ such that (after replacing $\psi_2$ with $-\psi_2$ if necessary) we have
\be \label{ineq:signs}
\psi_2(r_i) > 0 \ \ \mbox{for} \ \ i=1,3 \quad \mbox{and} \quad \psi_2(r_i) < 0 \ \ \mbox{for} \ \ i=2,4.
\ee
Now let $\Psi_2 = \cE_{a} \psi_2$ with $a=1-2s$ be the extension of $\psi_2$ to $\R^{N+1}_+$. Since $\Psi_2 \in C^0(\overline{\R^{N+1}_+})$ with $\Psi_2(x,0) = \psi_2(x)$, we deduce from \eqref{ineq:signs} that $\Psi_2$ has at least two nodal domains on $\R^{N+1}_+$. Hence, by Proposition \ref{prop:nodal} below, we conclude that $\Psi_2$ has exactly two nodal domains in $\R^{N+1}_+$, which we denote $\Omega_+$ and $\Omega_-$ in what follows. 

Now we use the radial symmetry of $\psi_2=\psi_2(|x|)$ on $\R^N$, which implies cylindrical symmetry of $\Psi_2 = \Psi_2(|x|,t)$ on $\R^{N+1}_+$ with respect to the $t$-axis. Clearly, the nodal domains $\Omega_+$ and $\Omega_-$ are cylindrically symmetric sets with respect to the $t$-axis. Therefore, it suffices to consider the set $\mathcal{N} = \{ r \geq 0 \} \times \{ t > 0\}$ and likewise let $\Omega_{\pm}^{\mathrm{rad}} = \{ (|x|,t) : (x,t) \in \Omega_{\pm} \}$ be the corresponding nodal domains of $\Psi_2$ on $\mathcal{N}$ regarded as a function of $r=|x|$ and $t$. By continuity of $\Psi_2$, we deduce that
\be \label{eq:connect}
(r_i, \eps) \in \Omega_{+}^{\mathrm{rad}} \ \ \mbox{for} \ \ i=1,3 \quad \mbox{and} \quad (r_i, \eps) \in \Omega_{-}^{\mathrm{rad}} \ \ \mbox{for} \ \ i=2,4,
\ee
for all $\eps \in (0,\eps_0)$, where $\eps_0 > 0$ is some sufficiently small constant. Furthermore, note that the sets $\Omega_{\pm}^{\mathrm{rad}} \subset \mathcal{N}$ must be arcwise connected. From this fact and \eqref{eq:connect} we conclude that there exist two injective continuous curves $\gamma_+, \gamma_- \in C^0([0,1]; \overline{\mathcal{N}})$ with $\overline{\mathcal{N}} = \{ r \geq 0 \} \times \{ t \geq 0 \}$ with the following properties.
\begin{itemize}
\item $\gamma_+(0) = (r_1,0), \gamma_+(1) = (r_3,0)$ and $\gamma_+(t) \in \Omega_{+}^{\mathrm{rad}}$ for $t \in (0,1)$.
\item $\gamma_-(0) = (r_2,0), \gamma_-(1) = (r_4,0)$ and $\gamma_-(t) \in \Omega_{-}^{\mathrm{rad}}$ for $t \in (0,1)$.
\end{itemize}
Since $r_1 < r_2 < r_3 < r_4$ holds, we conclude that the curves $\gamma_+$ and $\gamma_-$ intersect in $\mathcal{N}$; e.\,g., this follows from applying \cite[Lemma D.1]{FrLe-10}. But this contradicts $\Omega_{+}^{\mathrm{rad}} \cap \Omega_{-}^{\mathrm{rad}} = \emptyset$. \end{proof}

\section{Proof of Theorem \ref{thm:linear2}}

\label{sec:proofthm2}
Let $N \geq 1$ and $s \in (0,1)$. Consider the $H=(-\DD)^s + V$ acting on $L^2_{\mathrm{rad}}(\R^N)$, where $V$ satisfies the assumptions (V1) and (V2) in Section \ref{sec:linear}. By assumption, the operator $H$ has at least two negative eigenvalues $\Er_1 < \Er_2 < \inf \sigma_{\mathrm{ess}}(H)$ below the essential spectrum. 

For notational convenience, we let $\psi(r)=\psi_2(r)$ denote the second radial eigenfunction of $H=(-\DD)^s +V$ for the rest of this section. 

Since $H$ is self-adjoint, we have the orthogonality $(\psi, \psi_1)= 0$, where $\psi_1=\psi_1(r) > 0$ is (up to a sign) the unique positive ground state eigenfunction of $H$ (see Lemma \ref{lem:perron} below). Thus $\psi(r)$ has to change its sign at least once on $(0,+\infty)$. On the other hand, by Proposition \ref{prop:osc} above, we conclude that $\psi(r)$ changes its sign exactly once or exactly twice on the half-line $(0, +\infty)$. To rule out the latter possibility, we use a continuation argument for the second radial eigenfunction of a suitable family of self-adjoint operators $\{ H_\kappa \}_{\kappa \in [0,1]}$ such that $H_{\kappa=0} = (-\DD)^s + V$, whereas $H_{\kappa=1}=-\DD + W$ is a classical Schr\"odinger operator with $W$ being some attractive Gaussian potential to ensure that $-\DD+W$ has to at least two negative radial eigenvalues. Before turning to the actual proof of Theorem \ref{thm:linear2}, we work out the preliminaries of this continuation argument first.

\subsection{Continuation of Eigenfunctions} 
Recall that the radial potential $V=V(r)$ satisfies the conditions (V1) and (V2). Without loss of generality, we can assume that $V(r) \to 0$ as $r \to +\infty$ in what follows. Hence, by assumption, the operator $H=(-\DD)^s +V$ acting on $L^2_{\mathrm{rad}}(\R^N)$ has at least two radial simple negative eigenvalues $\Er_1 < \Er_2 < \mathrm{inf} \, \sigma_{\mathrm{ess}}(H) = 0$. Since we shall employ a continuation argument in $s$, it turns out to be convenient to denote 
\be
H_0=(-\DD)^{s_0} + V
\ee
for the operator given in the assumptions of Theorem \ref{thm:linear2}.

For $\kappa \in [0,1]$, we introduce the following family of self-adjoint operators $\{ H_\kappa \}_{\kappa \in [0, 1]}$ acting on $L^2_{\mathrm{rad}}(\R^N)$ given by
\be \label{def:Hfamily}
H_\kappa = \left \{ \begin{array}{ll} H^{(1)}_{3\kappa} & \quad \mbox{for $\kappa \in [0,1/3]$,} \\
H^{(2)}_{3 \kappa-1} & \quad \mbox{for $\kappa \in (1/3,2/3]$,} \\
H^{(3)}_{3 \kappa-2} & \quad \mbox{for $\kappa \in (2/3,1]$} . \end{array} \right .
\ee
Here the sub-families $\{ H^{(i)}_{\tau} \}_{\tau \in [0,1]}$, with $i=1,2,3$, act on $L^2_{\mathrm{rad}}(\R^N)$ and are defined as
\be \label{eq:Hk1}
H^{(1)}_{\tau} = (-\Delta)^{s_0} + V +  \tau W \ \ \mbox{for} \ \ \tau \in [0,1],
\ee
\be \label{eq:Hk2}
H^{(2)}_\tau= (-\Delta)^{s_0} + (1-\tau) V + W \ \ \mbox{for} \ \ \tau \in [0,1],
\ee
\be \label{eq:Hk3}
H_\tau^{(3)} = (-\Delta)^{ (1-\tau) s_0 + \tau} + W \ \ \mbox{for} \ \ \tau \in [0,1].
\ee
Here $W$ denotes the attractive Gaussian potential
\be
W(x) = -g e^{-x^2} ,
\ee
where $g> 0$ is the  universal constant taken from Lemma \ref{lem:bigg} above. Note that $H_{\kappa=0} = H_0$ and $H_{\kappa=1} = -\DD + W$. By Lemma \ref{lem:bigg}, the operator $H_{\kappa} = -\DD + W$ has at least two radial negative eigenvalues. We have the following result.

\begin{lemma} \label{lem:deform}
Let the family $\{ H_\kappa \}_{\kappa \in [0,1]}$ acting on $L^2_{\mathrm{rad}}(\R^N)$ be defined as above. Then, for every $\kappa \in [0,1]$, each $H_\kappa$ has at least two negative radial eigenvalues and the first two radial eigenvalues $\Er_{1,\kappa} < \Er_{2,\kappa} < 0$ are simple.  Furthermore, let $\psi_{\kappa} \in L^2_{\mathrm{rad}}(\R^N)$ with $\| \psi_{\kappa } \|_{L^2}=1$ denote the radial eigenfunction of $H_\kappa$ for the second eigenvalue $\Er_{2,\kappa}$. Then, after possibly changing the sign of $\psi_{\kappa}$, the following properties hold.
\begin{enumerate}
\item[(i)] $\Er_{2, \kappa'} \to \Er_{2,\kappa}$ as $\kappa' \to \kappa$.
\item[(ii)] $\psi_{\kappa'} \to \psi_{\kappa}$ in $L^2 \cap L^\infty_{\mathrm{loc}}$ as $\kappa' \to \kappa$.
\end{enumerate}
\end{lemma} 

\begin{remark*} {\em
The previous lemma can be obviously generalized to the first radial eigenvalue $\Er_{1,\kappa}$ and its corresponding eigenfunction $\psi_{1,\kappa}$ (as well as higher eigenvalues and eigenfunctions if present.) However, we shall only need the second eigenfunction/eigenvalue in the proof of Theorem \ref{thm:linear2} below. }
\end{remark*}

\begin{proof}
The proof of Lemma \ref{lem:deform} is provided in Appendix \ref{app:eigencont} below.
\end{proof}

\subsection{Completing the Proof of Theorem \ref{thm:linear2}}
Let $\{ H_{\kappa} \}_{\kappa \in [0,1]}$ denote the family of operators defined above. For notational convenience, we use the following notation
$$
H_{\kappa} = (-\Delta)^{s_\kappa} + V_\kappa .
$$ 
Recall that $H_{\kappa=0} = (-\Delta)^{s} + V$ and $H_{\kappa=1} = -\Delta -g e^{-x^2}$. Let $\psi_\kappa$ denote the second normalized radial eigenfunction of $H_\kappa$ and let $E_\kappa= \Er_{2,\kappa}$ denote the corresponding second radial eigenvalue of $H_\kappa$. We have the following properties.
\begin{itemize}
\item[(i)] For $\kappa \in [0,1)$, the function $\psi_\kappa(r)$ changes its sign exactly once or twice on $(0,+\infty)$.
\item[(ii)] The function $\psi_{\kappa=1}$ changes its sign exactly once on $(0,+\infty)$.
\item[(iii)] $\psi_{\kappa'} \to \psi_{\kappa}$ in $L^2 \cap L^\infty_{\mathrm{loc}}$ as $\kappa' \to \kappa$.
\end{itemize}
Indeed, property (i) follows from observations already made at the beginning of this section. Also, we deduce (ii) from classical ODE arguments, since $\psi_{\kappa=1}(r)$ is the second radial eigenfunction of the Schr\"odinger operator $H_{\kappa=1}=-\Delta - g e^{-x^2}$. Property (iii) is given by Lemma \ref{lem:deform} above.

Suppose now that $\psi(r) = \psi_{\kappa=0}(r)$ changes its sign  exactly twice on $(0,+\infty)$. We define 
\be \label{def:kappastar}
\kappa_* := \sup \big \{ \kappa \in [0,1) : \mbox{$\psi_{\kappa}(r)$ changes its sign exactly twice on $(0,+\infty)$} \big \}. 
\ee  
From properties (i) and (iii), we deduce that if $\psi_{\kappa}(r)$ changes its sign twice, then $\psi_{\kappa+\eps}(r)$ changes its sign twice for $\eps>0$ small. In particular, this shows that $\kappa_* >0$ holds. Furthermore, we conclude  that 
\be
\mbox{$\psi_{\kappa_*}(r)$ changes its sign exactly once.}
\ee
If otherwise $\psi_{\kappa_*}(r)$ changed its sign exactly twice, we would get a contradiction from the previous observation and the definition of $\kappa_*>0$.

Note that $\psi_{\kappa_*} \in L^1$ by Lemma \ref{lem:eigendecay} (i) if $\kappa_* < 1$ and  from standard arguments for classical Schr\"odinger operators if $\kappa_*=1$. Now, we claim that
\be \label{eq:intpsi}
\int_{\R^N} \psi_{\kappa_*} \, dx= 0,
\ee
and 
\be \label{eq:intVpsi}
\int_{\R^N} V_{\kappa_*} \psi_{\kappa_*} \, dx = 0.
\ee
For the moment, let us assume that these identities hold. By combining them, we find
\be \label{eq:contradict}
\int_{\R^N} \left \{ V_{\kappa_*}(|x|) - V_{\kappa_*}(r_*) \right \} \psi_{\kappa_*}(x) \,dx = 0,
\ee
where $r_*>0$ is such that $\psi_{\kappa_*}(r) \geq 0$ for $r \in [0,r_*)$ and $\psi_{\kappa_*}(r) \leq 0$ for $r \in [r_*,+\infty)$ (and $\psi_*$ does not vanish identically on both intervals). But since $V_{\lambda_*}$ is monotone increasing, we obtain a contradiction from \eqref{eq:contradict}. Thus everything is reduced to proving \eqref{eq:intpsi} and \eqref{eq:intVpsi}. 

We begin with the proof of \eqref{eq:intVpsi}. Since $\psi_\kappa(r)$ changes its sign twice in $(0,+\infty)$ for $0 \leq \kappa < \kappa_*$, we can assume that, for some $0 < r_{1,\kappa} < r_{2,\kappa} < +\infty$, 
\be \label{eq:sign1}
\mbox{$\psi_\kappa(r) \geq 0$ on $[0,r_{1,\kappa})$}, \ \ \mbox{$\psi_\kappa(r) \leq 0$ on $[r_{1,\kappa},r_{2,\kappa})$,}  \ \ \mbox{$\psi_\kappa(r) \geq 0$ on $[r_{2,\kappa}, +\infty)$},
\ee
and $\psi_\kappa(r)$ does not vanish identically on each of these intervals. Since $\psi_\kappa \in C^0$ for all $\kappa \in [0,1]$ and $\psi_{\kappa} \to \psi_{\kappa_*}$ in $L^\infty_{\mathrm{loc}}$ as $\kappa \to \kappa_*$ and $\psi_{\kappa_*}(0) \neq 0$ by Theorem 1, we see that $r_{1,\kappa} \not \to 0$ as $\kappa \to \kappa_*$. Since $\psi_{\kappa_*}(r)$ changes its sign exactly once, we conclude that we must have $r_{2, \kappa} \to +\infty$ as $\kappa \to \kappa_*$. Thus, for some $0 < r_{1,\kappa_*} < +\infty$,
\be  \label{eq:sign2}
\mbox{$\psi_{\kappa_*}(r) \geq 0$ on $[0,r_{1, \kappa_*})$}, \ \ \mbox{$\psi_{\kappa_*}(r) \leq 0$ on $[r_{1,\kappa_*}, +\infty)$},
\ee
where $\psi_{\kappa_*}(r)$ does not vanish identically on each of these intervals.

Next we note that $V_\kappa(r) \to 0$ as $r \to +\infty$ and $E_\kappa < 0$. Hence, by Lemma \ref{lem:eigendecay} (ii), the asymptotics  of $\psi_{\kappa}(r)$ for $\kappa \in [0,1)$ (and hence $s_\kappa < 1$) are given by
\be \label{eq:psi_asymp1}
\psi_{\kappa}(r) = - A_{\kappa}  \left (  \int_{\R^N} V_{\kappa} \psi_{\kappa} \, dx \right )  r^{-N-2s_{\kappa}} + o(r^{-N-2s_{\kappa}}) \ \ \mbox{as} \ \ r \to +\infty,
\ee 
with some positive constant $A_\kappa > 0$.  In view of \eqref{eq:psi_asymp1}, we deduce for $0 < \kappa_* \leq 1$ from \eqref{eq:sign1} that 
\be \label{ineq:Vstar}
\int_{\R^N} V_\kappa \psi_{\kappa} \,dx \leq 0 \ \ \mbox{for} \ \ \kappa \in [0, \kappa_*) ,
\ee
and 
\be \label{ineq:Vstar2}
\int_{\R^N} V_{\kappa_*} \psi_{\kappa_*} \,dx \geq 0 ,
\ee
which follows from \eqref{eq:psi_asymp1} and \eqref{eq:sign2} for $\kappa_* < 1$ and from Lemma \ref{lem:asym_luis} for $\kappa_*=1$.
Next, we note that
\be \label{eq:Vconv}
\int_{\R^N} V_{\kappa} \psi_\kappa \to \int_{\R^N} V_{\kappa_*} \psi_{\kappa_*} \ \ \mbox{as} \ \ \kappa \to \kappa_* .
\ee
Assuming this convergence, we conclude from \eqref{ineq:Vstar} and \eqref{ineq:Vstar2} that the claim \eqref{eq:intVpsi} holds. Hence it remains to prove \eqref{eq:Vconv}. We discuss the cases $\kappa_* \leq 2/3$ and $\kappa_* >2/3$ separately as follows. 

First, assume that $\kappa_* \in (0,2/3]$ holds. In this case, we have $s_\kappa = s_{\kappa_*}=s_0$ is constant for all $\kappa \leq \kappa_*$. Moreover, we have that $E_{\kappa} \leq -\lambda < 0$ with some constant $\lambda> 0$, by continuity of $\kappa \mapsto E_\kappa$ and the negativity $E_\kappa <0$. Also, we readily see that $V_{\kappa}(x) + \lambda \geq 0$ for all $|x| \geq R$, where $R \geq 0$ is some constant independent of $\kappa$. Thus we can apply Lemma \ref{lem:eigendecay} to deduce the uniform decay estimate $|\psi_{\kappa}(x)| \lesssim \langle x \rangle^{-N-2s_0}$ for $\kappa \in [0,2/3]$. Since $\psi_\kappa \to \psi_{\kappa_*}$ in $L^\infty_{\mathrm{loc}}$, this uniform decay bound implies that $\psi_\kappa \to \psi_{\kappa_*}$ in $L^1$. By the fact that $V_{\kappa} \to V_{\kappa_*}$ in $L^\infty$, we deduce that \eqref{eq:Vconv} holds, provided that $\kappa_* \leq 2/3$. 

Assume now that $\kappa_* \in (2/3,1]$. Here we simply note that $H_{\kappa} = (-\Delta)^{s_\kappa} + W$ for $\kappa \in (2/3,1]$, where the fixed potential $W=-g e^{-x^2}$ is smooth and  rapidly decaying. Recalling that $\psi_{\kappa} \to \psi_{\kappa_*}$ in $L^2$ and $V_\kappa = W \in L^2$ for $\kappa \in (2/3,1]$, we directly obtain \eqref{eq:Vconv} in this case.

It remains to prove \eqref{eq:intpsi}. Indeed, we integrate the equation for $\psi_{\kappa_*}$ over $\R^N$. This gives us 
\be
\int_{\R^N} (-\Delta)^{s_{\kappa_*}} \psi_{\kappa_*} \, dx + \int_{\R^N} V_{\kappa_*} \psi_* \,dx = E_{ \kappa_*} \int_{\R^N} \psi_{\kappa_*} \, dx.
\ee
Note that $\int_{\R^N} (-\Delta)^{s_{\kappa_*}} \psi_{\kappa_*} \,dx=  0$ holds, which follows from the Fourier inversion formula and the fact that $(-\DD)^{s_{\kappa_*}} \psi_{\kappa_*} \in L^1$. Recalling that $E_{\kappa_*} \neq 0$, we infer from \eqref{eq:intVpsi} that \eqref{eq:intpsi} also holds. This proves \eqref{eq:contradict} and leads to the desired contradiction. 

The proof of Theorem \ref{thm:linear2} is now complete. \hfill $\square$





\section{Nondegeneracy of Ground States}

\label{sec:nondeg}
This section is devoted to the proof of Theorem \ref{thm:nondeg} that establishes the nondegeneracy of ground states $Q \geq 0$ for equation \eqref{eq:Q}. By Proposition \ref{prop:Q},  we can assume that $Q=Q(|x|) > 0$ is radial without loss of generality. This proof of Theorem \ref{thm:nondeg}  will be divided into two main steps as follows. First, we establish the triviality of the kernel of the linearized operator $L_+$ in the space of radial functions. Here the oscillation result of Theorem \ref{thm:linear2} enables us to follow the ideas of \cite{FrLe-10} given for $N=1$ space dimension. Furthermore, to rule out further elements in the kernel of $L_+$ apart from $\partial_{x_i} Q$, with $i=1, \ldots, N$, we decompose the action of $L_+$ using spherical harmonics. In fact, this latter argument is based in spirit on an argument by Weinstein for this nondegeneracy of ground states for NLS in \cite{We-85}. In our setting, we need certain technical adaptations to the fractional Laplacian using heat kernel and Perron--Frobenius arguments, which are worked out in Appendix \ref{app:misc}.

\subsection{Nondegeneracy in the Radial Sector}
First, we show that the restriction of $L_+$ on radial functions has trivial kernel.

\begin{lemma} \label{lem:kernelrad}
We have $(\mathrm{ker} \, L_+) \cap L^2_{\mathrm{rad}}(\R^N) = \{ 0 \}$.
\end{lemma}

\begin{proof}
We argue by contradiction. Suppose that there is $v \in L^2_{\mathrm{rad}}(\R^N)$ with $v \not \equiv 0$ such that $L_+ v = 0$ holds. Recall that, by assumption, the Morse index of $L_+$ is one. Hence $0$ must be the second eigenvalue of $L_+$.  From Theorem \ref{thm:linear2} we conclude that (up to changing the sign of $v$) there is some $r_* > 0$ such that
\be \label{eq:rstar}
v(r) \geq 0 \ \ \mbox{for $r \in [0, r_*)$}, \quad v(r) \leq 0 \ \ \mbox{for $r \in [r_*,+\infty)$},
\ee
and $v \not \equiv 0$ on both intervals $[0,r_*)$ and $[r_*, +\infty)$. Now, this fact puts us in the same situation, as if Sturm oscillation theory was applicable to the radial eigenfunctions of $L_+$. Therefore we can follow the strategy of \cite{FrLe-10} based on the nondegeneracy proof for NLS ground states in \cite{ChGuNeTs-08}. First, we note that a calculation shows that
\be \label{eq:LplusR}
L_+ Q = -\alpha Q^{\alpha+1} \quad \mbox{and} \quad L_+ R = -2s Q,
\ee
where 
\be
R = \frac{2s}{\alpha} Q + x \cdot \nabla Q.
\ee
Using the decay and regularity estimates for $Q$, it is easy to check that $R \in H^{2s+1}(\R^N)$ and hence $Q$ and $Q^{\alpha+1}$ both belong to $\mathrm{ran} \, L_+$. However, the strict monotonicity of $Q$ together with $v  \perp \mathrm{ran} \, L_+$ leads to a contradiction as follows. For any $\mu \in \R$, we have orthogonality
\be \label{eq:vortho}
(v, Q^{\alpha+1} - \mu Q) = 0 .
\ee
But by choosing now $\mu_* = (Q(r_*))^\alpha$ with $r_* > 0$ from \eqref{eq:rstar}, the fact that $Q(r)>0$ is monotone decreasing implies that $v(r) (Q^{\alpha+1}(r) - \mu_* Q(r)) \geq 0$ (with $\not \equiv 0$) for all $r> 0$. But this contradicts \eqref{eq:vortho} and completes the proof of Lemma \ref{lem:kernelrad}. \end{proof}

\subsection{Nondegeneracy in the Non-Radial Sector}

Since $Q=Q(|x|)$ is a radial function, the operator $L_+=(-\Delta)^s + 1 - (\alpha+1) Q^\alpha$ commutes with rotations on $\R^N$. In what follows, let us assume that $N \geq 2$ holds. (The arguments can be adapted with some modifications to the case $N=1$; see \cite{FrLe-10} for the proof of the nondegeneracy result in the one-dimensional setting.) Using the decomposition in terms of spherical harmonics
\be
L^2(\R^N) = \bigoplus_{\ell \geq 0} \mathcal{H}_\ell ,
\ee
we find that $L_+$ acts invariantly on each subspace
\be 
\mathcal{H}_\ell = L^2(\R_+, r^{N-1} dr ) \otimes \mathcal{Y}_\ell.
\ee
Here $\mathcal{Y}_\ell = \mathrm{span} \, \{ Y_{\ell,m} \}_{m \in M_\ell}$ denotes space of the spherical harmonics of degree $\ell$ in space dimension $N$. Note that the index set $M_\ell$ depends on $\ell$ and $N$. Recall also that $-\Delta_{\mathbb{S}^{N-1}} Y_{\ell,m} = \ell (\ell + N-2) Y_{\ell,M}$.

For each $\ell \in \Lambda$, the action of $L_+$ on the radial factor in $\mathcal{H}_\ell$ is given by
\be
(L_{+,\ell} f)(r) = ( (-\Delta_{\ell})^s f) (r) + f(r) - (\alpha+1) Q^{\alpha}(r) f(r) ,
\ee
for $f \in C_0^\infty(\R_+) \subset L^2(\R_+, r^{N-1} dr)$. Here $(-\Delta_{\ell})^s$ is given by spectral calculus and the known formula 
\be \label{eq:DD_ell}
-\Delta_{\ell} = - \frac{\partial^2}{\partial r^2} - \frac{N-1}{r} \frac{\partial}{\partial r} + \frac{\ell(\ell+N-2)}{r^2} .
\ee
Technically, we consider \eqref{eq:DD_ell} as a self-adjoint operator in $L^2(\R_+)$ defined as the Friedrichs extension of the corresponding differential expression acting on $C_0^\infty(\R_+)$. 

Note that Lemma \ref{lem:kernelrad} above says that $\mathrm{ker} \, L_{+,0} = \{ 0 \}$. We now derive the following result.

\begin{lemma} \label{lem:kernelnonrad}
We have $\mathrm{ker} \, L_{+,1} = \mathrm{span} \, \{ \partial_r Q \}$ and $\mathrm{ker} \, L_{+,\ell} = \{ 0 \}$ for $\ell \geq 2$.
\end{lemma}

\begin{proof}
By differentiating equation \eqref{eq:Q}, we see that $L_+ \partial_{x_i} Q = 0$ for $i=1, \ldots, N$. Since $\partial_{x_i} Q = Q'(r) \frac{x_i}{r} \in \mathcal{H}_{\ell = 1}$, we deduce that $L_{+,1} Q' = 0$. Because of $Q'(r) < 0$ by Proposition \ref{prop:Q} and the Perron-Frobenius property of $L_{+,1}$ (see Lemma \ref{lem:perron}), we deduce that $0$ is the lowest and hence nondegenerate eigenvalue of $L_{+,1}$. Thus we conclude that $\mathrm{ker} \, L_{+,1} = \mathrm{span} \, \{ \partial_r Q \}$.

Finally, by Lemma \ref{lem:ordering}, we have the strict inequality $L_{+, \ell} > L_{+,1}$ in the sense of quadratic forms for any $\ell \geq 2$. Since $0$ is the lowest eigenvalue of $L_{+,1}$, this implies that $0$ cannot be an eigenvalue of $L_{+, \ell}$ for $\ell \geq 2$, which completes the proof of Lemma \ref{lem:kernelnonrad}. \end{proof}

\subsection{Completing the Proof of Theorem \ref{thm:nondeg}}
Let $\xi \in L^2(\R^N)$ satisfy $L_+ \xi = 0$. By Lemma \ref{lem:kernelrad} and \ref{lem:kernelnonrad}, we conclude that $\xi \in \mathcal{H}_{\ell=1}$ and that $\xi$ must be a linear combination of $\partial_{x_1} Q, \ldots, \partial_{x_n} Q$. \hfill $\square$

\section{Uniqueness of Ground States}

\label{sec:unique}

In this section, we give the proof of Theorem \ref{thm:unique}. Thanks to the nondegeneracy result of Theorem \ref{thm:nondeg}, we can now closely follow the strategy developed in \cite{FrLe-10}, where uniqueness of the ground state $Q=Q(r)> 0$ was shown for $N=1$ space dimension. That is, by starting from a given radial ground state $Q_{s_0}=Q_{s_0}(r)$ with $s_0 \in (0,1)$ given, we construct a branch $s \mapsto Q_s$ of radial ground state solutions to
\be \label{eq:Qflow}
(-\DD)^s Q_s + Q_s - |Q_s|^{\alpha} Q_s = 0 \ \ \mbox{in} \ \ \R^N.
\ee
The local existence and uniqueness of $Q_s$ for $s$ close to $s_0$ follows from an implicit function argument, based on the invertibility of the linearized operator $L_+$ around $Q_{s_0}$ in the radial sector. Then, by deriving a-priori bounds from above and below for $Q_s$ we will be able to extend the branch to $s =1$, linking the problem to the classical case, where uniqueness and nondegeneracy of $Q_{s=1}$ is well-known; see \cite{Kw-89,ChGuNeTs-08}. Finally, we show the uniqueness of the branch $Q_s$ starting from the ground state $Q_{s_0}$, which establishes the uniqueness result of $Q_{s_0}$.    

For the reader's convenience, the following presentation will be mostly self-contained. In contrast to \cite{FrLe-10}, the flow argument in $s$ will be carried out with fixed Lagrange multiplier in \eqref{eq:Qflow}. That is, the zeroth order term in \eqref{eq:Qflow} is $Q_s$ instead of $\lambda_s Q_s$ as in \cite{FrLe-10} with some function $\lambda_s$ depending on $s$. In fact, the approach with $\lambda_s \equiv 1$ taken here will turn out to be advantageous due to two reasons: First, the derivation of a-priori bounds will become more transparent. Second, the generalization of the flow argument for more general nonlinearities than pure-power nonlinearities will be more straightforward (to be addressed in future work).

\subsection{Construction of the Local Branch}
We start with some preliminaries. Let $N \geq 1$, $s \in (0,1)$, and $0 < \alpha < \alphcr(s_0,N)$ be fixed throughout the following. We define the real Banach space
\be
X^{\alpha} := \big \{ f \in L^2(\R^N) \cap L^{\alpha+2}(\R^N) : \mbox{$f$ is radial and real-valued} \big \} ,
\ee  
equipped with the norm 
\be
\| f \|_{X^\alpha} := \|f \|_{L^2} + \| f \|_{L^{\alpha+2}} .
\ee
For $s \in [s_0,1)$, we consider $Q \in X^\alpha$ that solve in the sense of distributions the equation
\be \label{eq:Qbranch}
(-\Delta)^s Q +  Q - |Q|^{\alpha} Q = 0 \quad \mbox{in $\R^N$}.
\ee
By a bootstrap argument, it is easy to see that indeed $Q \in H^{2s+1}(\R^N)$ holds (see, e.\,g., \cite[Lemma B.2]{FrLe-10} for $N=1$; for $N \geq 2$, the modifications of the proof given there are straightforward.) Nevertheless, we prefer to discuss solutions $Q$ in $X^\alpha$, since the space will be a natural $s$-independent choice when we below construct a local branch $Q_s$ parametrized by $s \in [s_0,1)$. Note also that, at this point, we do not assume that $Q \in X^\alpha$ is necessarily a positive solution of \eqref{eq:Qbranch}.  

\begin{prop} \label{prop:localQ}
Let $N \geq 1$, $0 < s_0 < 1$, and $0 < \alpha < \alphcr(s_0,N)$. Suppose that $Q_0 \in X^\alpha$ solves equation \eqref{eq:Qbranch} with $s=s_0$. Moreover, assume that the linearization $L_+ = (-\Delta)^{s_0} + 1 - (\alpha+1) |Q_0|^{\alpha}$ has trivial kernel on $L^2_{\mathrm{rad}}(\R^N)$.

Then, for some $\delta > 0$, there exists a map $Q \in C^1(I; X^\alpha)$ defined on the interval $I=[s_0, s_0 + \delta)$ such that the following holds, where we denote $Q_s=Q(s)$  in the sequel.
\begin{enumerate}
\item[(i)] $Q_s$ solves \eqref{eq:Qbranch} for all $s \in I$.
\item[(ii)] There exists $\eps > 0$ such that $Q_s$ is the unique solution of \eqref{eq:Qbranch} for $s \in I$ in the neighborhood $\{ Q \in X^\alpha : \| Q-Q_0 \|_{X^\alpha} < \eps \}$. In particular, we have that $Q_{s=0} = Q_0$ holds.
\end{enumerate}
\end{prop}

\begin{remarks*} {\em
1.) By standard arguments, the operator $L_+$ has a bounded inverse on $X^\alpha$ if $L_+$ has trivial kernel on $L^2_{\mathrm{rad}}(\R^N)$. 

2.) We could also construct a local branch $Q_s$ for $s \in (s_0-\tilde{\delta}, s_0]$ with some $\tilde{\delta} > 0$ small. But since we are ultimately interested in extending the branch $Q_s$ to $s=1$, we shall only consider the case $s \geq s_0$.

3.) Recall that, in contrast to \cite{FrLe-10}, we do not introduce a Lagrange multiplier function $\lambda_s$ depending on $s$ in \eqref{eq:Qbranch}. 
}
\end{remarks*}

\begin{proof}
As in \cite{FrLe-10}, we use an implicit function argument for the map
\be
F : X^\alpha \times [s_0, s_0 +\delta) \to X^\alpha, \quad F(Q,s) = Q - \frac{1}{(-\Delta)^s +1} |Q|^{\alpha} Q .
\ee
Clearly, we have that $F(Q, s) = 0$ if and only if $Q \in X^\alpha$ solves \eqref{eq:Qbranch} and, by assumption, we have $F(Q_0,s_0)=0$. Moreover, following the arguments in \cite{FrLe-10}, we can show that $F$ is well-defined map  of class $C^1$. Next, we consider the Fr\'echet derivative 
\be
\partial_Q F(Q_0,s_0) = 1 + K, \quad \mbox{with $\displaystyle K=- \frac{1}{(-\Delta)^{s_0} + 1} (\alpha+1) |Q_0|^{\alpha}$}.
\ee
Note that the operator $K$ is compact on $L_{\mathrm{rad}}(\R^N)$ and we have that $-1 \not \in \mathrm{\sigma}(K)$, which follows from the fact that $0$ is not an eigenvalue of $L_+ = (-\Delta)^{s_0}+1 - (\alpha+1) |Q_0|^\alpha$. Moreover, we check that $K$ maps $X^\alpha$ to $X^\alpha$ and that the bounded inverse $(1+K)^{-1}$ exists on $X^\alpha$. Hence the Fr\'echet derivative $\partial_Q F$ has a bounded inverse on $X^\alpha$ at $(Q_0,s_0)$. By the implicit function theorem, we deduce that the assertions (i) and (ii) of Proposition \ref{prop:localQ} hold for some $\delta > 0$ and $\eps >0$ sufficiently small. \end{proof}

\subsection{A-Priori Bounds and Global Continuation}

We now turn to the global extension of the locally unique branch $Q_s$, with $s \in [s_0,s_0+\delta)$, constructed in Proposition \ref{prop:localQ} above. Again, we follow the general strategy derived in \cite{FrLe-10} for the $N=1$ case. However, since we work with fixed Lagrange multiplier in \eqref{eq:Qflow}, the arguments below will differ from those in \cite{FrLe-10}.

For the following discussion, we first recall that $N \geq 1$, $0 < s_0 < 1$, and $0 < \alpha < \alphcr(s_0,N)$ are fixed. Furthermore, we suppose that $Q_0 \in X^{\alpha}$ is fixed and satisfies the assumptions of Proposition \ref{prop:localQ}. Correspondingly, let denote $Q_s \in C^1(I; X^\alpha)$ with $I = [s_0, s_0 + \delta)$ the unique local branch given by Proposition \ref{prop:localQ}. We consider the {\em maximal extension} of the branch $Q_s$ with $s \in [s_0, s_*)$, where $s_*$ is given by
\begin{align}
s_*  := \sup \big \{ s_0 < \tilde{s} < 1  & :  \mbox{$Q_s \in C^1([s_0, \tilde{s}); X^\alpha)$ and $Q_{s}$ satisfies the assumptions }  \\
& \mbox{of Proposition \ref{prop:localQ} for $s \in [s_0,\tilde{s})$} \big \} . \nonumber
\end{align}
Our ultimate goal will be to show that $s_*=1$ holds if $Q_{s_0} =Q_0$ is a ground state solution of \eqref{eq:Qbranch}. To do this, we derive a-priori bounds along the maximal branch $Q_s$ with $s \in [s_0,s_*)$. Before we proceed, we introduce some shorthand notation for the rest of this section.

\begin{convention*}
We write $X \lesssim Y$ to denote that $X \leq CY$ with some constant $C > 0$ that only depends on $N, \alpha, s_0,$ and $Q_0$. As usual, the constant $C> 0$ may change from line to line. Furthermore, we use $X \sim Y$ to denote that both $X \lesssim Y$ and $Y \lesssim X$ hold.
\end{convention*}

\noindent
As essential step, we derive the following a-priori bounds. 
\begin{lemma} \label{lem:apriori}
We have the a-priori bounds
$$
\int_{\R^N} | Q_s|^2 \sim \int_{\R^N} | (-\Delta)^{s/2} Q_s |^2 \sim \int_{\R^N} | Q_s |^{\alpha+2}  \sim 1
$$
for all $s \in [s_0,s_*)$. 
\end{lemma}

\begin{proof}
We divide the proof into the following three steps.

\subsubsection*{Step 1. Lower bounds}
It is convenient to use the following notation
\be
M_s = \int_{\R^N} |Q_s|^2, \quad T_s = \int_{\R^N} |(-\Delta)^{s/2} Q_s|^2, \quad V_s = \int_{\R^N} |Q_s|^{\alpha+2} ,
\ee
for $s \in [s_0,s_*)$.  We first claim that
\begin{equation}
 \label{eq:equiv}
M_s \sim T_s \sim V_s.
\end{equation}
To see this, we integrate equation \eqref{eq:Qbranch} against $Q_s$ and $x\cdot \nabla Q_s$, respectively. (Note that, by some straightforward estimates, we have $x\cdot \nabla Q_s\in H^s$.) This yields
\be \label{eq:poho1}
T_s+M_s=V_s .
\ee
Moreover, by using the fact that $[\nabla\cdot x,(-\Delta)^s]=-2s (-\Delta)^s$, we obtain the Pohozaev-type identity
\be \label{eq:poho2}
\frac{N-2s}{2} T_s + \frac N 2 M_s = \frac{N}{\alpha+2} V_s .
\ee
Combining \eqref{eq:poho1} and \eqref{eq:poho2}, some elementary computations lead to \eqref{eq:equiv}.

Next, from the fractional Gagliardo--Nirenberg--Sobolev inequality \eqref{ineq:GN} we derive that
$$
 \frac{T_s^{\frac{N \alpha }{4s}} M_s^{\frac{\alpha+2}{2}- \frac{N \alpha}{4s}}}{V_s} \gtrsim 1,
$$
which follows from the fact that the optimal constant in \eqref{ineq:GN} can be uniformly bounded with respect to $s \geq s_0$ with $s_0$ fixed; see, e.\,g., Lemma A.4 in \cite{FrLe-10} and its proof for $N=1$ dimension; the adaptation to $N \geq 2$ poses no difficulties. In view of \eqref{eq:equiv}, we see that 
\be
M_s \sim T_s \sim V_s \gtrsim 1,
\ee
 i.\,e., the quantities are uniformly bounded away from zero. Thus it remains to find an upper bound for one of the quantities $M_s$, $T_s$ or $V_s$.

\subsubsection*{Step 2. Regularity} 
We claim that
\begin{equation}
 \label{eq:reg1}
\| (-\DD)^{s-\frac{\alpha N}{4(\alpha+2)}} Q_s \|_{L^2}^2 \lesssim V^{\frac{2(\alpha+1)}{\alpha+2}} \,.
\end{equation}
(This is a regularity statement since $s-\alpha N/4(\alpha+2)>s/2$ for $s \geq s_0$, thanks to the condition $\alpha < \alphcr(s_0,N)$. However, the point here is that the constant here is independent of $s$.) Indeed, from the identity $Q_s = ((-\DD)^s + 1)^{-1} |Q_s|^{\alpha} Q_s$ and Plancherel's identity, we deduce that
\begin{equation} \label{ineq:young1}
\| (-\DD)^t Q_s \|_{L^2} = \left \| \frac{(-\DD)^t}{(-\DD)^s + 1} |Q_s|^{\alpha} Q_s \right \|_{L^2} 
\leq \| (-\DD)^{t-s} ( |Q_s|^{\alpha} Q_s) \|_{L^2}
\end{equation}
for any $t \geq 0$. In particular, the choice
\be \label{eq:t_regular}
t := s - \frac{N \alpha }{4(\alpha+2)}
\ee
satisfies $s > t > s-s_0/2 \geq s_0/2$ thanks to the condition $\alpha < \alphcr(s_0,N) \leq \alphcr(s,N)$ for $s \geq s_0$. For this choice of $t$, the operator $(-\DD)^{t-s}$ on $\R^N$ is given by convolution with $|x|^{-N(\alpha+4)/(2 (\alpha+2))}$ up to a multiplicative constant depending only on $\alpha$ and $N$. Hence, by the weak Young inequality,
\be \label{ineq:tbound}
\| (-\DD)^{t} Q_s \|_{L^2}^2 \lesssim \| |x|^{-\frac{N(\alpha+4)}{2(\alpha+2)} } \ast ( |Q_s|^{\alpha} Q_s ) \|_{L^2}^2 \lesssim  \| |Q_s|^{\alpha+1} \|^2_{L^{\frac{\alpha+2}{\alpha+1}}} = V^{\frac{2(\alpha+1)}{\alpha+2}} \,,
\ee
which is the claimed bound.

\subsubsection*{Step 3. Upper bounds} Recall that is suffices to derive an upper for one of the quantities $M_s$, $T_s$, or $V_s$.
We shall prove a uniform upper bound for $V_s$ as follows. If we differentiate the equation satisifed by $Q_s$ with respect to $s$, we get
$$
L_{+,s} \dot Q_s = - (-\DD)^s \log ( -\DD) Q_s
$$
with $\dot{Q}_s = \frac{d Q_s}{ds}$ and $L_{+,s} = (-\DD)^s - (\alpha+1) |Q_s|^{\alpha}+1$. Using this fact and the identity $L_{+,s} Q_s =-\alpha Q_s^{\alpha +1}$, we find
\be \label{eq:law_Vs}
\frac{d }{ds} V_s = (2+ \alpha) \left ( Q_s^{1+\alpha}, \dot{Q}_s \right ) =  \frac{2 + \alpha}{\alpha}  \left ( Q_s , (-\DD)^s \log (-\DD) Q_s \right ).
\ee
Next, for $t$ defined in \eqref{eq:t_regular} above, we have that
$$
\begin{aligned}
& \left ( Q_s, (-\DD)^s \log(-\DD) Q_s \right )  = 2 \int_{\R^N} |\xi|^{2s} \left ( \log |\xi| \right )  |\widehat{Q}_s(\xi)|^2 \, d \xi \\
& = 2 \int_{|\xi| \leq R}  |\xi|^{2s} \left ( \log |\xi| \right )  |\widehat{Q}_s(\xi)|^2 \, d \xi + 2 \int_{|\xi| > R}  |\xi|^{2s} \left ( \log |\xi| \right )  |\widehat{Q}_s(\xi)|^2 \, d \xi \\
& \leq 2 (\log R) \int_{\R^N} |\xi|^{2s} |\widehat{Q}_s(\xi)|^2 \, d \xi + 2 R^{2s-4t} (\log R) \int_{\R^N} |\xi|^{4t} |\widehat{Q}_s(\xi)|^2 \, d \xi \\
& \leq 2  ( \log R) T_s + 2 R^{2s-4t} (\log R) V_s^{ \frac{2(\alpha+1)}{\alpha+2}} ,
\end{aligned} 
$$
provided that we take $R \geq e^{\frac{1}{4t - 2s}}$, which gives us $|\xi|^{2s-4t} (\log |\xi|) \leq R^{2s-4t} (\log R)$ for $|\xi| \geq R$. Note also that we used \eqref{ineq:tbound} in the last step. Now, we choose $$R^{4t-2s} = c V_s^{\frac{2 (1+\alpha)}{2+\alpha}-1} \geq e=2.71...$$ with some suitable constant $c \sim 1$. This is possible thanks to the uniform lower bound $V_s \gtrsim 1$ and since $\frac{2 (1+\alpha)}{2+\alpha} - 1  >0$. Going back to \eqref{eq:law_Vs} and recalling that $T_s \sim V_s$, we obtain the differential inequality
\be
\frac{d}{ds} V_s \lesssim \left  (1 + \log V_s \right ) V_s,
\ee
which by integration yields the uniform upper bound $V_s \lesssim e^{e^s} \lesssim 1$ for $s \in [s_0, s_*)$. 
The proof of Lemma \ref{lem:apriori} is now complete.
 \end{proof}

\begin{lemma} \label{lem:Quniform}
Suppose that $Q_{s_0}(|x|) > 0$ is positive. Then, for all $s \in [s_0,s_*)$, we have
$$
Q_s(|x|) > 0 \ \ \mbox{for} \ \ x \in \R^N, \quad Q_s(|x|) \lesssim |x|^{-N} \ \ \mbox{for} \ \ |x| \gtrsim 1.
$$
\end{lemma}

\begin{remark*} {\em 
By Proposition \ref{prop:Q}, we have the improved bounds $C_1 \langle x \rangle^{-N-2s} \lesssim_{N,s} Q_s(x) \lesssim_{N,s} C_2 \langle x \rangle^{-N-2s}$. However, the point here is to obtain bounds that are uniform in $s$.}
\end{remark*}

\begin{proof}
The  positivity of $Q_s(|x|) > 0$ for $s \in [s_0, s_*)$, provided that $Q_{s_0}(|x|) > 0$ initially holds, can be inferred from spectral arguments. That is, we note that $Q_{s}$ is the ground state of the linear operator $A_s = (-\DD)^{s} + 1 - |Q|^{\alpha}$. For $s=s_0$, this follows from the assumption that $Q_{s_0}(x)> 0$ holds and that $A_s$ enjoys a Perron--Frobenius property; see Lemma \ref{lem:perron}. Since $A_{s'} \to A_s$ as $s' \to s$ in the norm resolvent sense, we deduce that $Q_s$ is the ground state eigenfunction of $A_s$ for $s \in [s_0, s_*)$. Hence we deduce that $Q_s(x) > 0$ for $s \in [s_0, s_*)$. See \cite[Lemma 5.5]{FrLe-10} for details in $N=1$ dimensions, but the generalization to $N \geq 2$ is straightforward.  

Once the positivity of $Q_s(x) > 0$ is established, we now derive the uniform decay bound 
\be \label{ineq:Qsdecay}
Q_s(|x|) \lesssim |x|^{-N} \ \ \mbox{for} \ \ |x| \gtrsim 1.
\ee
First, recall that $Q_s \in L^1$ holds by Proposition \ref{prop:Q}. Integrating equation \eqref{eq:Q} over $\R^N$ and using the fact that $\int_{\R^N} (-\DD)^s Q_s = 0$ holds, we obtain the identity
\be
\int_{\R^N} Q_s  = \int_{\R^N} Q^{\alpha+1}_s .
\ee
By H\"older's inequality, we have that $\int_{\R^N} Q_s^{\alpha+1} \leq \left ( \int_{\R^N} Q_s  \right )^{\frac{1}{1 + \alpha}} \left ( \int_{\R^N} Q_s^{\alpha+2} \right )^{\frac{\alpha}{(1+\alpha)^2}}$. Using the a-priori bound from Lemma \ref{lem:apriori}, we deduce the uniform bound
\be \label{ineq:Qs_L1}
\int_{\R^N} Q_s \lesssim 1.
\ee 
By Proposition \ref{prop:Q}, the function $Q_s(x) > 0$ is decreasing in $|x|$. Hence, for any $R > 0$,
\be
\int_{\R^N} Q_s \geq \int_{|x| \leq R} Q_s(R) \gtrsim R^N Q_s(R) .
\ee 
In view of \eqref{ineq:Qs_L1} we conclude that \eqref{ineq:Qsdecay} holds, completing the proof of Lemma \ref{lem:Quniform}. \end{proof}

We conclude this subsection with the following convergence fact.

\begin{lemma} \label{lem:Qconv}
Let $(s_n)_{n=1}^\infty \subset [s_0,s_*)$ be a sequence such that $s_n \to s_*$ and suppose that $Q_{s_n}(|x|) > 0$ for $n \in \mathbb{N}$. Then, after possibly passing to a subsequence, we have $Q_{s_n} \to Q_*$ in $L^2(\R^N) \cap L^{\alpha+2}(\R^N)$ as $n \to +\infty$. Moreover, the function $Q_*(|x|) > 0$ is positive and satisfies
\be \label{eq:Qstar}
(-\DD)^{s_*} Q_* + Q_*  - Q_*^{\alpha+1} = 0 \ \ \mbox{in} \ \ \R^N.
\ee 
\end{lemma}

\begin{proof}

Let $Q_n = Q_{s_n}$ in the following. Recall the uniform bound $\| Q_{n} \|_{H^{s_0}} \lesssim 1$ by the a-priori bounds in Lemma \ref{lem:apriori}. Passing to a subsequence if necessary, we have $Q_n \weakto Q_*$ weakly in $H^{s_0}$. Furthermore, by local Rellich compactness, we can assume that $Q_n \to Q_*$ in $L^2_{\mathrm{loc}}$ and pointwise a.\,e.~in $\R^N$. Thanks to the uniform decay bound in Lemma \ref{lem:Quniform}, we can upgrade this to $Q_n\to Q_*$ strongly in $L^2$. Next, the condition $\alpha < \alphcr(s_0,N)$ and Sobolev embeddings ensure that $\|Q_n \|_{L^p} \lesssim 1$ with some $p > \alpha +2$. Thus, by H\"older's inequality, we deduce that $Q_{n} \to Q_*$ in $L^{\alpha+2}(\R^N)$. 

Finally, we note that the limit $Q_*(|x|) > 0$ is a positive function in $L^2(\R^N) \cap L^{\alpha+2}(\R^N)$ and satisfies equation \eqref{eq:Qstar}. Indeed, since $Q_{n}(|x|) > 0$ and $Q_n \to Q_*$ pointwise a.\,e.~on $\R^N$, we deduce that $Q_*(|x|) \geq 0$. Furthermore, thanks to the uniform lower bounds in Lemma \ref{lem:apriori} and $Q_n \to Q_*$ in $L^{2}$, we obtain that $Q_* \not \equiv 0$. Moreover, by passing to the limit in the equation satisfied by $Q_n$, we deduce that \eqref{eq:Qstar} holds. But this shows that $Q_* = ((-\DD)^{s_*} + 1)^{-1} Q^{\alpha+1}_*$. Since the kernel of $((-\Delta)^{s_*} + 1)^{-1}$ is positive (a classical fact for $s_*=1$; for $s_* <1$, see Lemma \ref{lem:Greenunif}), we obtain that $Q_*(|x|) > 0$ as well.  \end{proof}

We conclude this subsection by showing that $s_*=1$ holds, if the branch $Q_s$ starts at a ground state.

\begin{lemma} \label{lem:Qglobal}
Let $Q_{s_0} =Q_{s_0}(|x|)> 0$ be a ground state solution of \eqref{eq:Q}. Then its maximal branch $Q_s$ with $s \in [s_0, s_*)$ extends to $s_*=1$.
\end{lemma}

\begin{proof}
By Lemma \ref{lem:Quniform}, we have $Q_s(|x|) > 0$ for all $s \in [s_0,s_*)$. We consider the linearized operators along the branch 
\be
L_{+,s} = (-\DD)^s + 1 - (\alpha+1) Q_s^{\alpha}.
\ee
By an adaptation of the arguments in \cite{FrLe-10}, we see that the Morse index of $L_{+,s}$ acting on $L^2_{\mathrm{rad}}(\R^N)$ is constant, i.\,e.,
\be
\mathcal{N}_{-, \mathrm{rad}}(L_{+,s}) = 1, \quad \mbox{for $s \in [s_0,s_*)$}.
\ee 
Consider now a sequence $s_n \to s_*$ with $s_n < s_*$ and let $Q_{n} = Q_{s_n}$. By Lemma \ref{lem:Qconv}, we see that $Q_n \to Q_*$ in $L^{2} \cap L^{\alpha+2}$ and $Q_*(|x|)> 0$ satisfies equation \eqref{eq:Qstar}. Suppose now $s_* < 1$ holds. We claim that $Q_*$ is a  ground state solution for \eqref{eq:Qstar}, which would yield a contradiction, since the nondegeneracy result in Theorem \ref{thm:nondeg} would imply that the branch $Q_s$ can be extended beyond $s_*< 1$.

That $Q_*$ is a ground state solution can be directly seen as follows. Let $L_{+,*}$ denote its corresponding linearization 
\be
L_{+,*} = (-\DD)^{s_*} + 1 - (\alpha+1) Q^\alpha_*.
\ee
Following the discussion in \cite{FrLe-10}, we conclude that $L_{+,s_n} \to L_{+,*}$ in the norm resolvent sense. By the lower semi-continuity of the Morse index with respect to this convergence, this shows
\be
\liminf_{n \to \infty} \mathcal{N}_{-, \mathrm{rad}}(L_{+,n}) \geq \mathcal{N}_{-, \mathrm{rad}}(L_{+,*}) .
\ee
On the other hand, if we integrate equation \eqref{eq:Qstar} against $Q_*$, we see that $(Q_*,L_{+,*} Q_*) = - \alpha \int |Q_*|^{\alpha+2} <0$.  Thus, by the min-max principle, we conclude that $L_{+,*}$ acting on $L^2_{\mathrm{rad}}(\R^N)$ has Morse index equal to one. Hence $Q_*(|x|) > 0$ is a ground state solution to \eqref{eq:Qstar} in the sense of Definition \ref{def:Qground}. Therefore $L_*$ has trivial kernel on $L^2_{\mathrm{rad}}(\R^N)$ by Theorem \ref{thm:nondeg}. In particular, the branch $Q_s$ could be extended beyond $s_*$, if $s_* < 1$ was true. Hence $s_*=1$ holds. \end{proof}

\subsection{Completing the Proof of Theorem \ref{thm:unique}}
Let $N \geq 1$, $0 < s_0 < 1$, and $0 < \alpha < \alphcr(s_0,N)$ be fixed. Suppose that $Q_{s_0} = Q_{s_0}(|x|) >0$ and $\tilde{Q}_{s_0} = \tilde{Q}_{s_0}(|x|) > 0$ are two radial positive ground states for equation \eqref{eq:Q} with $s=s_0$ and $Q_{s_0} \not \equiv \tilde{Q}_{s_0}$.

By Theorem \ref{thm:nondeg}, we conclude that $Q_{s_0} \in X^{\alpha}$ and $\tilde{Q}_{s_0} \in X^{\alpha}$ both satisfy the assumptions of Proposition \ref{prop:localQ}.  By Lemma \ref{lem:Qglobal}, both branches extend to $s=1$; i.\,e., we have $Q_s \in C^1([s_0,1); X^\alpha)$ and $\tilde{Q}_s \in C^1([s_0, 1); X^\alpha)$. Moreover, we deduce that $Q_s \not \equiv \tilde{Q}_s$ for all $s \in [s_0, 1)$ from the local uniqueness in Proposition \ref{prop:localQ}. Furthermore, by Lemma \ref{lem:Qconv} and $s_*=1$, we deduce that $Q_{s} \to Q_*$ and $\tilde{Q}_s \to \tilde{Q}_*$ in $L^2 \cap L^{\alpha+2}$ as $s \to 1$ with $s < 1$. However, it known that uniqueness holds for the positive radial solution $Q_*(|x|) >0$ in $L^2 \cap L^{\alpha+2}$ solving
\be
-\Delta Q_* + Q_* - Q_*^{\alpha+1} = 0 \ \ \mbox{in} \ \ \R^N .
\ee
For this result, see Kwong \cite{Kw-89} for $N \geq 2$, whereas for $N=1$ the unique solution $Q_*$ is in fact known in closed form (see, e.\,g., \cite{FrLe-10}). (Note that a direct bootstrap argument shows that $Q \in C^2$ holds and thus the result \cite{Kw-89} is applicable.) Therefore, we have $Q_* \equiv \tilde{Q}_*$ and hence $\| Q_s - \tilde{Q}_s \|_{L^2 \cap L^{\alpha+2}} \to 0$ as $s \to 1$ with $s < 1$. But it is also known that $Q_*$ has nondegenerate linearization $L_* = -\Delta + 1 - (\alpha+1) Q_*^{\alpha}$. In particular, the operator $L_* =-\Delta + 1 -(\alpha+1) Q_*^{\alpha}$ is invertible on $L^2_{\mathrm{rad}}(\R^N)$; see \cite{Kw-89,ChGuNeTs-08}. Thus, by an implicit function argument in the spirit of the proof of Proposition \ref{prop:localQ}, there exists a unique branch $\underline{Q}_s \in C^1((1-\eps,1]; X^\alpha)$ solving \eqref{eq:Qbranch} around $\underline{Q}_{s=1} =Q_*$ with some small $\eps> 0$, which contradicts $Q_s \not \equiv \tilde{Q}_s$ for all $s \in [s_0,1)$. 

The proof of Theorem \ref{thm:unique} is now complete. \hfill $\square$

\begin{appendix}

\section{Continuation of Eigenfunctions}

\label{app:eigencont}

In this section, we prove some technical results needed in the proof of Theorem \ref{thm:linear2}. 

\begin{lemma} \label{lem:bigg}
Let $N \geq 1$. There is some constant $g > 0$ such that, for any $s \in (0,1]$, the operator $H=(-\Delta)^s - g e^{-x^2}$ has at least two negative radial  eigenvalues.
 \end{lemma}

\begin{proof}
The variational principle (see, e.\,g., \cite[Theorm 12.1]{LiLo-97}) says that if  there exist two radial, orthonormal functions $\psi_1,\psi_2\in H^s(\R^N)$ such that the two-by-two matrix
$$
\left( (\psi_j,(-\Delta)^s\psi_k) + (\psi_j, W \psi_k) \right)_{1 \leq j,k \leq 2}
$$
has two negative eigenvalues, then also the operator $(-\Delta)^s+W$ has two negative radial eigenvalues and they are bounded from above by the eigenvalues of the matrix. 

We take the radial functions
$$
\psi_1(x) = \pi^{-1/4} e^{-x^2/2} \,,\qquad \psi_2(x) = \left( N \pi^{N/2} /2 \right)^{-1/2} (x^2 - N/2) e^{-x^2/2} \,.
$$
These two functions are orthonormal. (They are the first two radial eigenfunctions of the harmonic oscillator in $N \geq 1$ dimensions.) Consider the matrices $T^{(s)}=(t_{jk}^{(s)})$ and $V^{(\eps)}=(v_{jk}^{(\eps)})$ with
$$
t^{(s)}_{jk} = (\psi_j,(-\Delta)^s\psi_k) \,,
\qquad
v^{(\eps)}_{jk} = (\psi_j, e^{-\eps x^2} \psi_k)
$$
and note that $t^{(s)}_{jk}= t^{(s)}_{kj}$ and $v^{(\eps)}_{jk}=v^{(\eps)}_{kj}$. By orthonormality, the matrix $V^{(\eps=0)}$ is the identity matrix. Since the matrix $V^{(\eps)}$ is symmetric and depends continuously on $\eps$, for all $\eps>0$ sufficiently small,  both eigenvalues are strictly positive. From now on, let this value of $\eps$ be fixed. We denote the corresponding positive eigenvalues of $V^{(\eps)}$ by $0<\mu_0\leq \mu_1$.

Next, we observe the simple fact that the entries of $T^{(s)}$ are uniformly bounded with respect to $s\in (0,1]$. That is, there is a constant $C>0$ such that $\sup_{1 \leq i,j \leq k} |t^{(s)}_{jk}|\leq C$ for all $s\in (0,1]$. Hence the eigenvalues of $T^{(s)}-\tilde g V$ lie in the intervals $[-\tilde g \mu_j -C, -\tilde g \mu_j +C]$ for any $s\in(0,1]$. Thus both eigenvalues of $T^{(s)} - \tilde g V$ are negative for $\tilde g =2C/\mu_0$.

By the variational principle, we conclude that the operator $(-\Delta)^s -\tilde g e^{-\eps x^2}$ has two negative eigenvalues. By scaling, this means that $(-\Delta)^s - \eps ^{-s}\tilde g e^{-x^2}$ has two  negative eigenvalues. We may assume that $\eps \leq 1$ holds. Then, by the variational principle, the eigenvalues of $(-\Delta)^s - \eps^{-s}\tilde g e^{-x^2}$ are not smaller than those of $(-\Delta)^s - g e^{-x^2}$ with $g = \eps^{-1}\tilde g$. This completes the proof of Lemma \ref{lem:bigg}.
\end{proof}

\begin{lemma} \label{lem:asym_luis}
Let $N \geq 1$ and consider $H=-\DD+ W$ with $W(x)=-g e^{-x^2}$ and $g >0$ as in Lemma \ref{lem:bigg}. Let $\psi(r)$ denote the second radial eigenfunction of $H$ and assume that $\psi(r) \leq 0$ for $r \gg 1$. Then $\int_{\R^N} W \psi \, dx \geq 0$. 
\end{lemma}

\begin{proof}
From classical Sturm oscillation theory we obtain that $\psi(r)$ has exactly one zero at $r_* > 0$, say. By assumption on $\psi(r)$, we have that $\psi(r) > 0$ for $r < r_*$ and $\psi(r) < 0$ for $r > r_*$. Next, by integrating the equation $-\DD \psi + W \psi = E \psi$ over $\R^N$ and using that $\int_{\R^N} \DD \psi \, dx= 0$ (note that $\psi$ decays exponentially by standard arguments), we obtain that
\be \label{eq:luis1}
\int_{\R^N} W \psi \, dx = E \int_{\R^N} \psi \, dx .
\ee
Since $E < 0$, it suffices to show that 
\be \label{eq:luisfinal}
\int_{\R^N} \psi \, dx \leq 0.
\ee
Indeed, let $W_* = W(r_*)$.  Since $W(r)$ is monotone increasing, we have $W(r) \leq W_*$ for $r < r_*$ and $W(r) \geq W_*$ for $r > r_*$. Therefore $W(r) \psi(r) \leq W_* \psi(r)$ both for $r < r_*$ and $r > r_*$. Thus,
\be \label{eq:luis2}
 \int_{\R^N} W \psi \, dx \leq W_* \int_{\R^N} \psi \, dx .
\ee
Now, we claim that $E > W_*$. To show this, we note that $\psi$ satisfies the Dirichlet problem 
$$
-\DD \psi + (W-E) \psi = 0  \quad \mbox{on $\R^N \setminus B_{r_*}$}
$$
with the boundary conditions $\psi = 0$ on $\partial B_{r_*}$ and $\psi \to 0$ as $|x| \to +\infty$. By the maximum principle, we deduce that $\psi \not \equiv 0$  implies that $W(r)-E < 0$ for some $r > r_*$. Recalling that $W(r) \geq W_*$ for $r > r_*$, we conclude that $E > W_*$. Since $E \int_{\R^N} \psi \, dx \leq W_* \int_{\R^N} \psi \, dx$ by combining \eqref{eq:luis1} and \eqref{eq:luis2}, we find that
$$
(E-W_*) \int_{\R^N} \psi \, dx \leq 0.
$$
Since $E-W_* > 0$, we deduce that \eqref{eq:luisfinal} holds.
\end{proof}

\subsection{Proof of Lemma \ref{lem:deform}} First, we show that $H_\kappa$ has two radial and simple eigenvalues $\Er_{1,\kappa} < \Er_{2,\kappa} < 0$. For $s \in (0, 1)$, the simplicity of $\Er_{n,\kappa}$ follows from Theorem \ref{thm:linear2}. For $s=1$, the simplicity of radial eigenvalues of $H=-\Delta + W$ follows from classical ODE arguments. Hence it remains to show the existence of $\Er_{1,\kappa} < \Er_{2,\kappa} <0$. In view of \eqref{def:Hfamily}, it suffices to prove this fact for each of the families $H^{(i)}_\kappa$ with $i=1,2,3$ given in \eqref{eq:Hk1}--\eqref{eq:Hk3}.

Indeed, we note that $H^{(1)}_\kappa \leq H_0 = (-\Delta)^{s_0} + V$ for all $\kappa \in [0,1]$, since we have $W \leq 0$. By assumption on $H_0$ and the min-max principle, we conclude that $H^{(1)}_\kappa$ has at least two radial negative eigenvalues for all $\kappa \in [0,1]$ . Likewise, we see that $H^{(2)}_{\kappa} \leq (-\Delta)^{s_0} + W$ for all $\kappa \in [0,1]$ because of $V \leq 0$. By Lemma \ref{lem:bigg} and the min-max principle, we deduce that $H^{(2)}_{\kappa}$ has at least two radial negative eigenvalues for all $\kappa \in [0,1]$. Finally, we directly see from Lemma \ref{lem:bigg} that $H^{(3)}_{\kappa}$ has at least two radial negative eigenvalues for all $\kappa \in [0,1]$.

Next, we prove the properties (i) and (ii). Let $n \in \{1,2\}$ be fixed. For notational convenience, denote $E_\kappa = \Er_{2,\kappa}$ and $\psi_\kappa= \psi_{2,\kappa}$ in the following. (The proof below identically works for $\Er_{1,\kappa}$ and $\psi_{1,\kappa}$, but we do not need this here.)

First, we remark that property (i) (i.\,e.~continuity of eigenvalues) follows from standard spectral theory, since the self-adjoint operators converge $H_{\kappa'} \to H_{\kappa}$ in the norm resolvent sense as $\kappa' \to \kappa$. That is, for any $z \in \mathbb{C}$ with $\mathrm{Im} \, z \neq 0$, we have
\be \label{eq:normconv}
\left \| (H_{\kappa'} - z)^{-1} -  ( H_{\kappa} - z)^{-1} \right \|_{L^2 \to L^2} \to 0 \quad \mbox{as} \quad \kappa' \to \kappa.
\ee
We omit the straightforward details of the proof of this fact; see \cite{FrLe-10} for $N=1$ dimension.

To show property (ii), let $P_{\kappa} : L^2 \to L^2$ denote the corresponding projections onto the eigenspaces of $H_\kappa$ with discrete eigenvalues $E_{\kappa}$. By Riesz' formula,  we have 
\be \label{eq:Riesz}
P_{\kappa} = \frac{1}{2 \pi i} \oint_{\Gamma_{\kappa}} (H_\kappa - z)^{-1} \, dz ,
\ee
where $\Gamma_{\kappa}$ parameterizes some circle in $\mathbb{C}$ around $E_{\kappa} \in \R$ with radius $r > 0$ sufficiently small such that  $\{ z \in \mathbb{C} : |z - E_{\kappa} | \leq r \} \cap \sigma (H_\kappa) = \{ E_{\kappa} \}$.  From \eqref{eq:Riesz}, \eqref{eq:normconv} and property (i), we can deduce that $\| P_{\kappa'} - P_{\kappa} \|_{L^2 \to L^2} \to 0$ as $\kappa' \to \kappa$. Since $\mathrm{ran} ( P_{\kappa})$ is spanned by $\psi_{\kappa}$, it is easy to see that (after changing the sign of $\psi_{\kappa'}$ if necessary) that the $L^2$-operator convergence of the eigenprojections $P_{\kappa}$ imply  that the normalized eigenfunctions satisfy
\be \label{cv:L2}
\psi_{\kappa'} \to \psi_{\kappa} \ \ \mbox{in $L^2(\R^N)$} \ \ \mbox{as}  \ \ \kappa' \to \kappa.
\ee

Next, we note that $|E_\kappa| \leq C$ by property (i) and that $\| V_\kappa \|_{L^\infty} \leq C$ uniformly in $\kappa \in [0,1]$. Hence, by applying Proposition \ref{prop:hoeldereigen} for $\kappa \in [0,2/3]$ with $s=s_0$, we deduce 
\be \label{eq:Hoelderunif}
\| \psi_\kappa \|_{C^{0,\beta}} \leq C \ \ \mbox{for} \ \ \kappa \in [0,2/3],
\ee 
in the range $\kappa \in [0,2/3]$ with some constant $C>0$ independent of $\kappa$ and any $0 < \beta < 2 s_0$ fixed. For $\kappa \in (2/3,1]$, we recall that $H_\kappa=(-\DD)^{s_\kappa} + W$ with $s_\kappa \in [s_0, 1]$. Since the fixed potential $W \in \mathcal{S}(\R^N)$ belongs to the Schwartz class, we can easily bootstrap the equation $H_\kappa \psi_\kappa = E_{\kappa} \psi_\kappa$ to see that $\| \psi_\kappa \|_{H^m} \lesssim_m 1$ for any $m \geq 0$ and $\kappa \in [2/3,1]$. Choosing some fixed $m > N/2+1$, we deduce from Sobolev embeddings  that the uniform H\"older bound \eqref{eq:Hoelderunif} holds in fact for all $\kappa \in [0,1]$, by changing $C>0$ if necessary. From \eqref{cv:L2} and \eqref{eq:Hoelderunif}, we obtain the convergence
\be \label{cv:Linf}
\psi_{\kappa'} \to \psi_{\kappa} \ \ \mbox{in $L^\infty_{\mathrm{loc}}(\R^N)$} \ \ \mbox{as} \ \ \kappa' \to \kappa.
\ee
The proof of Lemma \ref{lem:deform} is now complete.  \hfill $\square$

\section{Regularity Estimates}

\subsection{H\"older Estimates}
Let $N \geq 1$ and $0 < s <1$ be fixed throughout the following. We consider the linear equation
\be \label{eq:linearapp}
(-\Delta)^s u + Vu = 0 \quad \mbox{in $\R^N$},
\ee
We assume that the potential $V : \R^N \to \R$ satisfies the regularity condition:
\be \label{cond:Vreg}
\mbox{$V \in L^\infty(\R^N)$ if $s > \frac 1 2$ and $V \in C^{0, \gamma}(\R^N)$ if $0 < s  \leq \frac 1 2$ with some $\gamma >1-2s$}. 
\ee
Concerning the linear equation \eqref{eq:linearapp}, we have the following regularity result.

\begin{prop} \label{prop:regu}
Suppose that $V$ satisfies \eqref{cond:Vreg}. If $u \in L^\infty(\R^N)$ solves \eqref{eq:linearapp}, then $u \in C^{1, \beta}(\R^N)$ with some $\beta \in (0,1)$.
\end{prop}

\begin{proof}
This fact is a direct consequence of Schauder-type estimates for $(-\Delta)^s$ derived in \cite{Si-07}. Indeed, we note that $(-\Delta)^s u = w$ with $w =-Vu$ in $L^{\infty}(\R^N)$. By \cite{Si-07}, this implies that $u \in C^{1, \beta}(\R^N)$ for any $\beta \in (0,2s-1)$, provided that $s \in ( \frac 1 2, 1)$ holds. It remains to consider the case $s \in (0,\frac 1 2]$. Here, we note  that $u \in C^{0,\beta}(\R^N)$ for any $\beta \in (0, 2s)$ if $s \in (0, \frac 1 2]$, by \cite{Si-07}. Therefore $(-\Delta)^s u = w$ with some $w \in C^{0,\alpha}(\R^N)$ with $\alpha = \min \{ \gamma, \beta \}$. Furthermore, by \cite{Si-07}, this yields the following.
\begin{itemize}
\item If $\alpha+ 2s \leq 1$, then $u \in C^{0, \alpha + 2s}(\R^N)$.
\item If $\alpha + 2s > 1$, then $u \in C^{1, \alpha + 2s -1}(\R^N)$.  
\end{itemize}
Since $\gamma > 1 -2s$ by assumption, we can repeat the above steps finitely many times to conclude that $u \in C^{1,\beta}(\R^N)$ for some $\beta \in (0,1)$. \end{proof}

We now turn to regularity properties and decay properties of the extension of $u$ to the upper half-space $\R^{N+1}_+$, which we still denote by $u$ for notational convenience.

\begin{prop} \label{prop:regu2}
Let $u \in L^\infty(\R^N)$ be as in Proposition \ref{prop:regu}  above. Then its extension $u=u(x,t)$ satisfies the following properties, where $C>0$ denotes some constant.
\begin{enumerate}
\item[(i)]  For some $0 < \beta < 2 \min \{s, 1-s\}$, 
$$\| u \|_{C^{0,\beta}(\overline{\RNup})} + \| \nabla_x u \|_{C^{0,\beta}(\overline{\RNup})} + \| t^{a} \partial_t u \|_{C^{0,\beta}(\overline{\RNup})} \leq C.$$
\item[(ii)] For all $x \in \R^{N}$ and $t >0$,
$$|\nabla_x u(x,t)| + |\partial_t u(x,t)| \leq \frac{C}{t} .$$
\item[(iii)] If $u(x,0) \to 0$ as $|x| \to +\infty$, then, for every $R >0$ fixed, 
$$
\| u \|_{L^\infty(B^+_R(x,0))}+ \| \nabla_x u \|_{L^\infty(B_R^+(x,0))} + \| t^a \partial_t u \|_{L^{\infty}(B^+_R(x,0))} \to 0 \ \ \mbox{as} \ \ |x| \to +\infty, 
$$

\end{enumerate}
\end{prop}

\begin{proof}
These results follow adapting the arguments in \cite[Proposition 4.6 and Lemma 4.8]{CaSi-10}. We omit the details.
\end{proof}

\subsection{$L^2$-Eigenfunction Estimates} Let $N \geq 1$ and $0 < s < 1$ be given. Suppose that $V \in L^\infty(\R^N)$ is a bounded potential and consider the fractional Schr\"odinger operator
\be
H= (-\Delta)^s + V.
\ee
We have the following H\"older estimate for $L^2$-eigenfunction of $H$. (The conditions on $V$ could be relaxed to unbounded potentials, but we do not need this here.)

\begin{prop} \label{prop:hoeldereigen}
If $u \in L^2(\R^N)$ solves $H u = Eu$ with some $E \in \R$, then $u \in C^{0,\beta}(\R^N)$ for any $0 < \beta < 2 s_0 \leq 2s$ and we have
$$
\| u \|_{C^{0,\beta}} \lesssim_{s_0,N,E, \| V \|_{L^\infty}}   \| u \|_{L^2},
$$
\end{prop}

\begin{remark*} {\em Since $V \in L^\infty(\R^N)$, we see that $u \in H^{2s}(\R^N)$. Therefore, if $2 s \geq 2s_0 > N/2$ then $u \in L^\infty(\R^N)$ by Sobolev embeddings. Moreover, by the H\"older estimates above, the result of Proposition \ref{prop:hoeldereigen} follows for $s_0 > N/4$. (A closer inspection of the proofs shows also the uniformity with respect to $s \geq s_0$). However, to deal with the range $0 < s_0 \leq N/4$, we have to use some refined and different arguments, which we provide in the proof given below. }
\end{remark*}

To prepare the proof of Proposition \ref{prop:hoeldereigen}, we first need the following local estimate.

\begin{lemma} \label{lem:localreg}
For $r > 0$, let $B_r= \{ x \in \R^N : | x | < r \}$. Suppose that $u \in L^2(\R^N)$ solves 
$$
(-\Delta)^s u + u = f \quad \mbox{in $B_2$},
$$
with some $f \in L^p(B_2)$ and $p \in [1, \infty)$. Then
$$
\| u \|_{L^q(B_{1})} \lesssim_{s_0,N,p,q}  \left ( \| f \|_{L^p(B_2)} + \| u \|_{L^2} \right ),
$$
for $q \in [1, p/(1-2sp/N))$ if $2sp \leq N$ and $q \in [1, +\infty]$ if $2sp > N$. 
\end{lemma}

\begin{proof}
Let $0 \leq \eta \leq 1$ be a smooth function with  $\eta(x) \equiv 1$ on $B_{1/2}$ and $\mathrm{supp} \, \eta \subset B_1$. Moreover, let $G_s$ denote the fundamental solution $\{ (-\Delta)^s + 1 \} G_s = \delta_0$ in $\R^N$. We claim that
\be \label{eq:luistrick}
 \{ (-\Delta)^s + 1 \} (\eta G_s ) = \delta_0 + \varphi_s,
\ee
where $\varphi_s$ is a (smooth) function satisfying the uniform bound
\be \label{ineq:varphiL2}
\| \varphi_s \|_{L^2} \lesssim_{s_0} 1,
\ee
for all $s \in [s_0,1)$. To prove \eqref{ineq:varphiL2}, we first note the $\delta_0$ occurs on the right side in \eqref{eq:luistrick} because of $\eta \equiv 1$ in a neighborhood of the origin. Clearly, we have that $\varphi_s = \{ (-\DD)^s +1 \} ( (\eta-1) G_s)$. From \cite{Si-07} we recall that $(-\DD)^s$ is a bounded map from $C^{1,1}$ to $C^{1,1-2s}$. (By inspecting the proof there, we see that the bound can be chosen uniform in $s \geq s_0 > 0$.) Clearly, the operator $(-\DD)^s+1$ enjoys the same properties, and hence we conclude that
\be \label{ineq:varphi1}
\| \varphi_s \|_{L^\infty} \leq \| \varphi_s \|_{C^{1,1-2s}} \lesssim_{s_0} \| (\eta-1) G_s \|_{C^{1,1}} \lesssim_{s_0} 1,
\ee
where in the last step we used the uniform bounds in Lemma \ref{lem:Greenunif}  together with the fact $1-\eta \equiv 0$ on $B_1$. Next, we claim the decay bound
\be \label{ineq:varphi2}
|\varphi_s(x)| \lesssim_{s_0} |x|^{-N-2s}. 
\ee
Since $\varphi_s \in L^\infty$, it suffices to derive this bound for $|x| > 2$. Indeed, by the singular integral representation for $(-\DD)^s$, we deduce for $x \not \in B_2$ that
\begin{align*}
| \left  (-\Delta)^s (\eta G_s) (x) \right | & = C_s \int_{B_2} \frac{\eta(y) G_s(y)}{|x-y|^{N+2s}} \, dy \\
& \leq \frac{C_s}{|x/2|^{N+2s}} \int_{B_2} \eta(y) G_s(y) \, dy \lesssim_{s_0} |x|^{-N-2s}.
\end{align*}
In the last step, we used that $C_s \lesssim_{s_0} 1$ holds together with the obvious fact that $\| G_s \|_{L^1} = 1$. Furthermore, we have the pointwise bound $|(\eta G_s)(x)| \lesssim_{s_0} |x|^{-N-2s}$ by Lemma \ref{lem:Greenunif}. Combining these decay bounds, we conclude that \eqref{ineq:varphi2} holds. Finally, we combine \eqref{ineq:varphi1} and \eqref{ineq:varphi2} to deduce the desired bound \eqref{ineq:varphiL2}. 

Now, we are ready to come the main point of the proof of Lemma \ref{lem:localreg}. Assume that $u \in L^2(\R^N)$ satisfies
\[ (-\DD)^s u + u = f \text{ in } B_{2}, \]
for some function $f \in L^p(B_{2})$ and some $p \in [1,\infty)$. In what follows, we set $f(x) \equiv 0$ for $|x| > 2$. For $x \in B_{1}$, we compute
\begin{align}
u(x) &= \left ( \left [ \{ (-\DD)^s + 1 \} (\eta G_s) \right] \ast u \right )(x) - ( \varphi_s \ast u)(x) \nonumber \\
&= \left ( \eta G_s \ast [\{ (-\DD)^s + 1 \} u] \right ) (x) - ( \varphi_s \ast u)(x) \nonumber \\
& = (\eta G_s \ast f)(x) - (\varphi_s \ast u)(x), \label{eq:luistrick2}
\end{align}
where we used that $\eta G_s$ is supported in $B_1$, and hence $\eta(x-y) G_s(x-y) f(y) = 0$ for $|x| \leq 1$ and $|y| > 2$.  

Recalling the bound \eqref{ineq:varphiL2}, we estimate the second term on right side in \eqref{eq:luistrick2} as follows:
\be
\| \varphi_s \ast u \|_{L^q(B_1)} \lesssim \| \varphi_s \ast u \|_{L^\infty} \lesssim_{s_0} \| u \|_{L^2}, 
\ee
using Young's (or H\"older's) inequality in the last step. For the first term on the right-hand side in \eqref{eq:luistrick2}, we note that $\| \eta G_s \|_m \lesssim_{s_0,m} 1$ for any $m \in [1, N/(N-2s))$ if $2s \leq N$ and $m \in [1,  +\infty]$ if $2s > N$. Thus, by Young's inequality, 
\be
\| \eta G_s \ast f \|_{L^q} \leq \| \eta G_s \|_{L^m} \| f \|_{L^p} \lesssim_{s_0,m} \| f \|_{L^p},
\ee
where
\be
\frac{1}{q} +1 = \frac{1}{m} + \frac{1}{p}.
\ee
Because of the range of $m$, we can obtain any $q \in [1, p/(1-2sp/N))$ if $2sp \leq N$ and $q \in [1, +\infty]$ if $2sp > N$.  The proof of Lemma \ref{lem:localreg} is now complete.\end{proof}

\begin{proof}[Proof of Proposition \ref{prop:hoeldereigen}]
Note that $Hu = Eu$ can be written as
\be
(-\DD)^s u + u = Wu
\ee
wiht $W = E-V + 1 \in L^\infty(\R^N)$. By iterating the estimate in Lemma \ref{lem:localreg}, we obtain that $u \in L^{q_k}_{\mathrm{loc}}$ for an increasing sequence $2=q_0 < q_1 < \cdots < q_n = +\infty$. Thus, after finitely many steps (bounded uniformly in $s \geq s_0$), we obtain that $u \in L^\infty_{\mathrm{loc}}(\R^N)$ with $\| u \|_{L^\infty(B_1(x_0))} \leq C$ for any $x_0 \in \R^N$, where $C > 0$ is independent of $x_0$. Therefore, we deduce 
\be
\| u \|_{L^\infty} \lesssim_{s_0, E, \| V \|_{L^\infty} }\| u \|_{L^2}.
\ee
Finally, from \cite{Si-07}, we see that $(-\Delta)^s u = g$ with $g \in L^\infty(\R^N)$ implies the H\"older bound
\be
\| u \|_{C^{0,\beta}} \lesssim_{s_0, E, \| V \|_{L^\infty}} \| u \|_{L^2},
\ee
for any $\beta < 2s_0 \leq 2s$, where the uniformity of these bounds for $s \geq s_0$ follows again from inspecting the proof in \cite{Si-07}.
\end{proof}

\section{Miscellanea for $H=(-\Delta)^s + V$}

\label{app:misc}

The purpose of this section is to derive regularity, decay, and asymptotic estimates for eigenfunctions of $H=(-\DD)^s +V$ that are uniform with respect to $s$ and $\| V \|_{L^\infty}$.

\subsection{Uniform Resolvent Bounds}
As a technical result, we first collect some uniform estimates for the kernel of the resolvent $((-\DD)^s + \lambda)^{-1}$ on $\R^N$ with $\lambda>0$.

\begin{lemma} \label{lem:Greenunif}
Suppose $N \geq 1$, $0< s < 1$, and $\lambda > 0$. Let $G_{s, \lambda} \in \mathcal{S}'(\R^N)$ denote the Fourier transform of $(| \xi |^{2s} + \lambda)^{-1}$. Then the following properties hold true.
\begin{enumerate}
\item[(i)] $G_{s,\lambda}(|x|) > 0$ is radial, positive, strictly decreasing in $|x|$, and smooth for $|x| \neq 0$.
\item[(ii)] For any multi-index $\nu \in \mathbb{N}^N$, the pointwise bound  
$$
| D_x^\nu G_{s,\lambda}(x) |  \lesssim_{s_0, \nu}  \lambda^{-1} |x|^{-N} \ \ \mbox{for} \ \ |x| > 0,
$$
holds uniformly for $s \in [s_0,1)$ with $s_0 \in (0,1)$ fixed.
\item[(iii)] It holds that
$$
\lim_{|x| \to +\infty} |x|^{N+2s} G_{s,\lambda}(x) = \lambda^{-2} C_{N,s}
$$
with some positive constant $C_{N,s} > 0$ depending only on $N$ and $s$.

\item[(iv)] $G_{s,\lambda} \in L^p(\R^N)$ for all $p \in [1,+\infty]$ with $1-1/p < 2s/N$. Moreover, we have $\| G_{s, \lambda} \|_{L^1}= \lambda^{-1}$.
\end{enumerate}
\end{lemma}



\begin{proof}
First, we note that
\be \label{eq:laplace}
((-\DD)^s+\lambda)^{-1} = \int_0^{+\infty} e^{-\lambda t} e^{-t (-\DD)^s } \,dt.
\ee
Furthermore, by the fact that the map $E \mapsto e^{- E^s}$ is completely monotone for $E \geq 0$ and $s \in (0,1)$ and by Bernstein's theorem (see \cite{FrLe-10} for details), we can write the fractional heat kernel $e^{-t(-\DD)^s}$ with $t> 0$ in terms of the subordination formula
\be \label{eq:subord}
e^{-t(-\DD)^s} = \int_0^{+\infty} \frac{1}{\sqrt{2 u}} e^{t^2 \Delta/(4 u)} \, d\mu_s(u),
\ee 
with some nonnegative finite measure $\mu_s \geq 0$ with $\mu_s \not \equiv 0$. From the known explicit formula for the Gaussian heat kernel $e^{t^2 \Delta}$ in $\R^N$ with $t >0$, we easily derive property (i).

To show (ii), we follow an argument used in \cite{FrLe-10} for $N=1$ dimensions. We first consider the case $\nu=0$. Let $p_s(t,x)$ with $\check{p}_s(t,x) = e^{-t |\xi|^{2s}}$ denote the heat kernel of $e^{-t (-\DD)^s}$ in $\R^N$. Recall that $p_s(t,x) > 0$ is positive. Let $s_0 \in (0,1)$ be fixed. We claim the uniform bound for $s \in [s_0,1)$
\be \label{ineq:psunif}
p_s(t,x) \lesssim_{s_0} \min \left \{ t^{-N/2s},  |x|^{-N} \right \}.
\ee
This can be seen as follows. By scaling, we have $p_s(t,x) = t^{-N/2s} p_s(1,t^{-1/2s} x)$. Thus it suffices to derive corresponding bounds for $p_s(1,x)$. Here, we first observe that 
\be
p_s(1,x) \lesssim \int_{\R^{N}} e^{-t |\xi|^{2s}} \, d \xi \lesssim_{s_0} 1.
\ee
Let $k=1, \ldots, N$ be fixed. Using Fourier inversion and integration by parts, we obtain
\begin{align*}
|x_k^N p_s(1,x) | & \lesssim \left | \int_{\R^N} e^{-|\xi|^{2s}} \partial_{\xi_k}^N e^{i \xi \cdot x} \, d \xi \right |  \lesssim  \int_{\R^N} \left | \partial_{\xi_k}^N e^{-|\xi|^{2s}}  \right | \, d \xi \\
& \lesssim \int_{|\xi|=0}^{+\infty} \left ( |\xi|^{2s-N} + \cdots + |\xi|^{2sN-N} \right ) e^{-|\xi|^{2s}} |\xi|^{N-1} d |\xi| \lesssim_{s_0} 1.
\end{align*} 
Therefore the upper bound $p_s(1,x) \lesssim_{s_0} |x|^{-N}$ holds, which completes the proof of \eqref{ineq:psunif} by scaling. Next, from \eqref{eq:laplace} and \eqref{ineq:psunif} we deduce
\be \label{ineq:Gssplit}
G_{s,\lambda}(x) \lesssim_{s_0} \left ( \int_{t \leq |x|^{2s}} e^{-\lambda t} |x|^{-N}  \,dt + \int_{t> |x|^{2s}} e^{-\lambda t} t^{-N/2s} \, dt \right ) \lesssim_{s_0} \lambda^{-1} |x|^{-N} .
\ee
Finally, we remark that estimating the derivatives $D_x^\nu G_{s,\lambda}$ with $|\nu| \geq 1$ follows in a similar fashion, by considering $\partial_x^\nu p_s(t,x)$ which corresponds to $i^{\nu} \xi^{\nu} \check{p}_s(t,\xi)$ on the Fourier side. We omit the details.

To show (iii), we recall from \cite{BlGe-60} that
\be \label{eq:Blumenthal}
\lim_{|x| \to +\infty} |x|^{N+2s} p_s(1,x) = C_{N,s}
\ee
with some positive constant $C_{N,s} > 0$. Thus, by scaling, we obtain $|x|^{N+2s} p_s(t,x) \to t C_{N,s}$ as $|x| \to +\infty$ by scaling. Thanks to the bounds \eqref{ineq:psunif} and dominated convergence, we conclude from equation \eqref{eq:laplace} that the  limit formula in (iii) holds. 

To prove (iv), we deduce from \eqref{eq:Blumenthal} and scaling that $0 < p_s(t,x) \leq 2 C_{N,s}$ for $t^{-1/2s} |x| \gtrsim R_s$ with some constant $R_s> 0$. Using this bound and the crude bound $p_s(t,x) \lesssim t^{-1/2s}$, we can show that $\| G_p \|_{L^p} < +\infty$, by using \eqref{eq:laplace} and splitting the $t$-integral into the regions $\{ t \leq R_s^{2s} |x|^{2s} \}$ and $\{ t \geq R_s^{2s} |x|^{2s} \}$ similarly to \eqref{ineq:Gssplit}.  Finally, since $G_{s,\lambda}$ is positive, we conclude that $\| G_{s,\lambda} \|_{L^1} = \int_{\R^N} G_{s,\lambda} = \check{G}_{s,\lambda}(0) = \lambda^{-1}$ by Fourier inversion. 
\end{proof}

\subsection{Asymptotics of Eigenfunctions}

The following result provides some uniform estimates regarding the spatial decay of eigenfunctions of $H=(-\Delta)^s + V$ below the essential spectrum. In fact, the following estimates can be found in the literature (see, e.\,g., \cite{CaMaSi-90}) without, however, no direct insight into uniformity of these estimates with respect to $s$ and $V$. 

\begin{lemma} \label{lem:eigendecay}
Let $N \geq 1$, $0 < s \leq 1$, and suppose that $V \in L^\infty(\R^N)$ with $V(x) \to 0$ as $|x| \to +\infty$. Assume that $u  \in L^2(\R^N)$ with $\| u \|_{L^2} = 1$ satisfies $(-\DD)^s u + V u = E u$ with some $E < 0$. Furthermore, let $0 < \lambda < -E$ be given and suppose that $R \geq 0$ is such that $V(x) + \lambda \geq 0$ for $|x| \geq R$. Then the following properties hold.
\begin{enumerate}
\item[(i)] For all $x \in \R^N$, it holds that
$$
|u(x)| \lesssim_{s_0, \lambda, R, \| V \|_{L^\infty}}  \langle x \rangle^{-N-2s}
$$
for $s \in [s_0,1)$ with $s_0 \in (0,1)$ fixed. 
\item[(ii)]  We have the asymptotic formula
$$
u(x) = - C \lambda^{-2} \cdot \left ( \int_{\R^N} V u \, dx \right ) |x|^{-N-2s} + o\left (|x|^{-N-2s} \right ) \ \ \mbox{as} \ \ |x| \to +\infty,
$$
where the positive constant $C =C (s,N)  >0$ is the constant from Lemma \ref{lem:Greenunif} {\em (iii)} above.
\end{enumerate}
\end{lemma}

\begin{proof} Let $s_0 \in (0,1)$ be fixed in the following. We start by showing part (i). From Lemma \ref{prop:hoeldereigen} we obtain the uniform bound 
\be \label{eq:unifdecay}
\| u \|_{L^\infty} \lesssim_{s_0, E, \| V \|_{L^\infty}} 1
\ee 
By assumption, we have $0 < \lambda < -E$ and $R \geq 0$ be such that $V(x) + \lambda \geq 0$ for $x \in B_R^c$. Furthermore, for any $f \in H^{2s}$, we have the general (Kato-type) inequality
\be \label{ineq:kato}
(-\DD)^s |f| \leq (\mathrm{sgn} \, f) (-\DD)^s f \ \ \mbox{a.\,e.~on $\R^N$},
\ee
where $(\mathrm{sgn} \, f)(x) = \bar{f}(x)/f(x)$ when $f(x) \neq 0$ and $(\mathrm{sgn}\, f)(x) = 0$ when $f(x) = 0$. Indeed, the estimate \eqref{ineq:kato} can be seen be an elementary argument using the singular integral formula for $(-\Delta)^s$ as follows. Note that \eqref{ineq:kato} is equivalent to
\be
\int_{\R^N} \frac{|f(x)| - |f(y)|}{|x-y|^{N+2s}} \, dy \leq (\mathrm{sgn} \, f)(x) \int_{\R^N} \frac{f(x)-f(y)}{|x-y|^{N+2s}} \, dy
\ee
for a.\,e.~$x \in \R^N$. But this inequality is easily seen to be equivalent to
\be
\int_{\R^N} \frac{[ 1- (\mathrm{sgn} \, f)(x) (\mathrm{sgn} \, f)(y) ] |f(y)|}{|x-y|^{N+2s} } \, dy \geq 0,
\ee
which immediately follows from the fact that $1- (\mathrm{sgn} \, f)(x) (\mathrm{sgn} \, f)(y) \geq 0$ for all $x,y \in \R^N$. This completes the proof of \eqref{ineq:kato}.


Now, we return to the proof of Lemma \ref{lem:eigendecay} itself. Since $V \in L^\infty$, we see that $u \in H^{2s}$. Hence by using \eqref{ineq:kato} on the sets $B_R^c \cap \{ u \geq 0 \}$ and $B_R^c \cap \{ u < 0 \}$ respectively, we deduce that
\be \label{ineq:absu1}
(-\DD)^s |u| + \lambda |u| \leq 0 \ \ \mbox{on $B_R^c$}.
\ee
Next, we claim that this implies
\be \label{ineq:absu2}
|u(x)| \lesssim_{s_0,\lambda, R, \| u \|_{L^\infty}} |x|^{-N-2s}  \ \ \mbox{on $\R^N$}.
\ee 
Indeed, this follows from a comparison argument as follows. Recall that $G_{s,\lambda}$ denotes the fundamental solution satisfying $( (-\Delta)^s+ \lambda ) G_{s,\lambda} = \delta_0$ in $\R^N$. Let $G_{s,\lambda}(x) \geq c>0$ for $|x| \leq R$ with $c=c(R,\lambda,s_0, N) >0$ the constant taken from Lemma \ref{lem:Greenunif} (v). Recall that $u \in L^\infty$. Now we choose $C_0= \| u \|_{L^\infty} c^{-1}$, which implies that $C_0 G_{s,\lambda}(x) \geq |u|(x)$ for $x \in B_R$ with some constant $C_0=C(s_0, \lambda, R,\| u \|_{L^\infty}) > 0$. Now we define the function 
\be
w:= C_0 G_{s,\lambda} - |u|,
\ee
 which is continuous away from the origin. Note that $w \geq 0$ on $B_R$ holds. We claim that $w \geq 0$ on $B_R^c$ as well. Suppose on the contrary that $w$ is strictly negative somewhere in $B_R^c$. Since $w \to 0$ as $|x| \to +\infty$ and $w \geq 0$ on $B_R$, this implies that $w$ attains a strict global minimum at some point $x_0 \in B_R^c$ with $w(x_0) <0$. By using the singular integral expression for $(-\DD)^s$, it is easy to see that $((-\DD)^s w)(x_0) < 0$. On the other hand, we have $(-\DD)^s w + \lambda w \geq 0$ on $B_R^c$, which implies that $((-\DD)^s w)(x_0) > 0$. This is a contradiction and we conclude that $w \geq 0$ on $\R^N$. From Lemma \ref{lem:Greenunif} (iv), we finally deduce that \eqref{ineq:absu2} holds. Combining \eqref{ineq:absu1} and \eqref{ineq:absu2}, we complete the proof of part (i).

To show part (ii), we argue as follows. Since $u = - ((-\DD)^s - E)^{-1} (V u)$, we can rewrite the equation for $u$ as an integral equation given by
\be
u = - G_{s,-E} \ast (u \psi) .
\ee
Observe that $|V(x) u(x)| \lesssim_{s_0,E, \| V \|_{L^\infty}} (1+|x|)^{-N-2s}$ because of (i) and $V \in L^\infty$. Since moreover $V \to 0$ as $|x| \to +\infty$, we deduce that $V(x) u(x) = o(|x|^{-N-2s})$ as $|x| \to +\infty$. Using the bounds and the asymptotic formula for $G_{s,-E}$ from Lemma \ref{lem:Greenunif}, we can apply Lemma \ref{lem:convolution} below with $\beta = N+2s >N$ to complete the proof of (ii). \end{proof}

The following auxiliary result was used in the previous proof.

\begin{lemma} \label{lem:convolution}
Let $k,f \in L^1(\R^N)$ satisfy $|k(x)| \leq C|x|^{-\beta}$ and $|f(x)| \leq C (1 +|x|)^{-\beta}$ with some constants $\beta > N$ and $C >0$. Moreover, assume that
$$
\lim_{|x|\to +\infty} |x|^\beta k(x) = K \quad\text{and}\quad \lim_{|x|\to +\infty} |x|^\beta f(x) = 0 \,.
$$
Then
$$
\lim_{|x|\to\infty} |x|^{\beta} \, (k*f)(x) = K \int_{\R^N} f(x)\,dx \,.
$$
\end{lemma}

\begin{proof}
Given $\eps > 0$, we can split $f=f_1+f_2$, where $f_1$ has compact support and $|f_2(x)|\leq \eps (1+|x|)^{-\beta}$. By dominated convergence, we have  $\lim_{|x|\to+\infty} |x|^{\beta} (k*f_1)(x) = K \int f_1(x)\,dx$. Thus, it suffices to prove that
$$
\limsup_{|x|\to+\infty} |x|^{\beta} |k*f_2(x)| \leq C\eps,
$$
with a constant $C>0$ depending only on $\| k \|_{L^1}$, $\beta$, and $N$. Because of the bound on $f_2$, this follows if we can prove that
$$
I=(1+|x|)^\beta \int_{\R^N}  |k(y)| (1+|x-y|)^{-\beta} \,dy \leq C \,.
$$
To see the latter bound, we split $I=I_1+I_2+I_3$, where $I_1$ corresponds to the integral restricted to $|x|\leq 2|x-y|$, $I_2$ to the region $|x|> 2|x-y|$ and $|y|\geq 1$ and, finally, $I_3$ to the remaining region. In the region corresponding to $I_1$, we have
$$
\frac{1+|x|}{1+|x-y|} \leq \frac{1+2|x-y|}{1+|x-y|} < 2,
$$
and therefore
$$
I_1 \leq 2^\beta \int_{\R^N} |k(y)| \, dy \leq C.
$$
On the other hand, in the regions corresponding to $I_2$ and $I_3$, we have $|x|\leq |x-y| + |y| < |x|/2 + |y|$ and therefore $|x|< 2|y|$. In the region of $I_2$, we use this in the form
$$
\frac{1+|x|}{|y|} < \frac{1+2|y|}{|y|} \leq 3,
$$
and deduce
$$
I_2 \leq 3^\beta \int_{|y|\geq 1}  (1+|x-y|)^{-\beta} \,dy \leq 3^\beta \int_{\R^N}  (1+|y|)^{-\beta} \,dy \leq C.
$$
Finally, in the region of $I_3$ we have $1+|x|< 1+2|y|<3$ and, trivially, $1+|x-y|\geq 1$. Therefore,
$$
I_3 \leq 3^\beta \int_{|y|<1}  |k(y)| \,dy \leq C.
$$
This completes the proof of Lemma \ref{lem:convolution}. \end{proof}

\subsection{Perron--Frobenius and Decomposition into Spherical Harmonics}

Recall that any function $u \in L^2(\R^N)$ can be decomposed using spherical harmonics as 
\be \label{eq:angular1}
u(x) = \sum_{\ell \in \Lambda} \sum_{m \in M_\ell} f_{\ell,m}(r) Y_{\ell,m} ( \Omega ), 
\ee
with $x=r \Omega$, $r =|x|$ and $\Omega \in \mathbb{S}^{N-1}$. Here $f_{\ell,m} \in L^2(\R_+, r^{N-1} dr)$ and $Y_{\ell,m} \in L^2(\mathbb{S}^{N-1})$ denotes the spherical harmonics of degree $\ell$ indexed by $m \in M_\ell$.  Note that the index set $\Lambda=\Lambda(N)$ satisfies $\Lambda(1) = \{0,1 \}$ and $\Lambda(N) = \N_0$ for $N \geq 2$. Likewise, the index set $M_\ell$ depends on $\ell$.  (In the one-dimensional case, the splitting into spherical harmonics corresponds to decomposition into odd and even functions on $\R$.)  In particular, if the sum in \eqref{eq:angular1} involves only terms with $\ell=0$, then the function $u=u(|x|)$ is radial (which means even if $N=1$).

Let us consider $H=(-\Delta)^s + V$ with $V \in K_s(\R^N)$ and $V=V(|x|)$ radial. Since $H$ commutes with rotations in this case, we can write the action of $H$ on functions $u$ in the form domain $H^s(\R^N) \subset L^2(\R^N)$ as 
\be
(H u)(x) = \sum_{\ell \in \Lambda} \sum_{m \in M_{\ell}} ( H_{\ell} f_{\ell,m})(r) Y_{\ell,m}(\Omega) .
\ee
Here $H_{\ell}$ acting on $L^2(\R_+, r^{N-1} dr)$ is given by
\be
H_\ell = ( -\Delta_\ell )^s + V,
\ee 
where $-\Delta_{\ell}$ is the Laplacian on $\R^N$ restricted to the sector of angular momentum $\ell$, which is known to be
\be
-\Delta_{\ell} = -\frac{\partial^2}{\partial r^2} - \frac{N-1}{r} \frac{\partial}{\partial r} + \frac{\ell(\ell+N-2)}{r^2} .
\ee

We have the following property, which is well-known in the classical case when $s=1$.
\begin{lemma} \label{lem:perron}
For each $\ell \in \Lambda$, the operator $H_{\ell}$ enjoys a Perron-Frobenius property. That is, if $E=\inf \sigma(H_{\ell})$ is an eigenvalue, then $E$ is simple and the corresponding eigenfunction can be chosen strictly positive.
\end{lemma}

\begin{proof}
From standard arguments, it suffices to show that the heat kernel $e^{-t H_{\ell}}$, with $t> 0$, is a positivity improving operator on $L^2(\R_+, r^{N-1} dr)$. (An operator $A$  is positivity improving if $Af>0$ is strictly positive whenever $f \geq 0$ is nonnegative and $f \not \equiv 0$.) Furthermore, we consider the higher dimensional case $N \geq 2$ in the following. See \cite{FrLe-10} for the proof in $N=1$ dimension.

Assume that $N \geq 2$ holds. First, we show that $e^{-t (-\Delta_{\ell})^s}$ is positivity improving on $L^2(\R_+, r^{N-1} dr)$. Indeed, from \eqref{eq:subord} and spectral calculus,  we obtain
\be \label{eq:subordinate2}
e^{-t (-\Delta_{\ell})^s } = \int_0^{+\infty} \frac{1}{\sqrt{2 \tau}} e^{t^2 \Delta_{\ell}/(4 \tau)} \, d\mu_s(\tau) ,
\ee
with some nonnegative measure $\mu_s \geq 0$. Thus it remains to show that $e^{t \Delta_{\ell}}$ is positivity improving. But adapting the arguments given in \cite{Le-09} for $N=3$ to general space dimensions, we see that the kernel of $e^{t \Delta_{\ell}}$ acting on $L^2(\R_+, r^{N-1} dr)$ is given by
\be
e^{t \Delta_{\ell}}(r,r') = \frac{c_N}{(4 \pi t)^{N/2}} \left ( \frac{r r'}{2t} \right )^{-\frac{N-2}{2}}   I_{\ell-\frac{N-2}{2}}\left ( \frac{r r'}{2t} \right )  e^{- \frac{r^2+r'^2}{4t} }.
\ee
Here $c_N > 0$ is some positive constant (depending only on $N$) and $I_\nu(x)$ denotes the modified Bessel function of the first kind. Since $I_\nu(x) > 0$ for all $x > 0$ and any index $\nu$, this manifestly shows that $e^{t \Delta_{\ell}}$ is positivity improving and hence the same property follows for $e^{-t (-\Delta_{\ell})^{s}}$ from the subordination formula \eqref{eq:subordinate2}.

Finally, we conclude that $e^{-t H_{\ell}}$ with $H=(-\Delta_{\ell})^s + V$ is positivity improving by a perturbation argument based on the Trotter product formula. We omit the standard details of this procedure. See \cite{ReSi-78}. \end{proof}

\begin{lemma} \label{lem:ordering}
Let $N \geq 2$ and $\ell' > \ell \geq 0$. Then we have strict inequality $H_{\ell'} > H_{\ell}$ in the sense of quadratic forms. In particular, if $E_{\ell'} = \inf  \sigma (H_{\ell'})$ and $E_{\ell} = \inf \sigma  (H_{\ell})$ are eigenvalues, then $E_{\ell'} > E_{\ell}$.
\end{lemma}

\begin{remark} {\em
The inequality $A > B$ means that $\mathcal{Q}(A) \subset \mathcal{Q}(B)$ and $(\phi,  A \phi ) > (\phi, B \phi)$ for all $\phi \in \mathcal{Q}(A)$ with $\phi \not \equiv 0$. Here $\mathcal{Q}(A)$ and $\mathcal{Q}(B)$ denote the quadratic form domains of $A$ and $B$, respectively.}
\end{remark}

\begin{proof}
Assume that $\ell' > \ell \geq 0$. We note the strict inequality $-\Delta_{\ell'} > -\Delta_{\ell} > 0$ in the sense of quadratic forms, which follows from $-\Delta_{(\ell')} - (-\Delta_{(\ell)}) = \delta/r^2 >0$ with some $\delta=\delta(\ell', \ell) > 0$ if $\ell' > \ell$. Next, for $0 < s < 1$ and $x > 0$, we recall the classical formula
\be
x^{s} = \frac{\sin(\pi s)}{\pi} \int_0^\infty \frac{x}{x+\lambda} \lambda^{s-1} \, d \lambda.
\ee
By spectral calculus, we deduce that
\be 
(-\Delta_{\ell'})^s - (-\Delta_{\ell})^s = \frac{\sin(\pi s)}{\pi} \int_0^\infty \left ( \frac{ -\Delta_{\ell'} }{-\Delta_{\ell'} + \lambda} - \frac{-\Delta_\ell}{-\Delta_{\ell}+\lambda}  \right  ) \, \lambda^{s-1} \, d \lambda >0 
\ee
in the sense of quadratic forms. Here we used the general fact that if $A > B \geq 0$, then $A/(A+\lambda) - B/(B+\lambda) > 0$ for any $\lambda > 0$, which can be seen  from the strict resolvent inequality $(A+\lambda)^{-1} < (B+\lambda)^{-1}$ for $\lambda > 0$ due to $A > B \geq 0$.

Adding the potential $V$ on both sides, we obtain that $H_{\ell'}> H_{\ell}$ in the sense of quadratic forms. The claim about the ordering of $E_{\ell'}$ and $E_{\ell}$ follows immediately.
\end{proof}

\section{Existence and Properties of Ground States}

\label{app:Q}

In this section, we provide some details of the proof of Proposition \ref{prop:Q}. Note that the existence of a ground  state can be inferred from

\begin{proof}[Proof of Proposition \ref{prop:Q}]
First, we prove part (i). In fact, we use a rather elementary proof (in spirit of \cite{We-83}) to show that the functional $J(u)$ has a minimizer. Denote $\alpha = \inf_{u \in H^s, u \not \equiv 0} J(u)$ in the following. Let $(u_n) \subset H^s(\R^N)$, with $u_n \not \equiv 0$, be a minimizing sequence for $J(u)$. By symmetric rearrangement, we have $J(f^*) \leq J(f)$ for any $f \in H^s$. Hence we can assume without loss of generality that $u_n=u_n^*$ holds. Moreover, by scaling, we can always normalize such that
\be
\| (-\DD)^{s/2} u_n \|_{L^2} = \| u_n \|_{L^2} = 1 \ \ \mbox{for all $n \geq 1$}.
\ee 
Note that the functions $u_n^*=u_n^*(|x|)$ are radial and monotone decreasing in $|x|$. Thus we deduce the uniform pointwise bound
\be \label{ineq:un_unif}
|u_n(x)| \lesssim |x|^{-N/2},
\ee
using also that $\|u_n\|_{L^2} \lesssim 1$ holds. By passing to a subsequence, we have that $u_n \weakto u_*$ in $H^s$ and $u_n \to u_*$ in $L^{\alpha+2}_{\mathrm{loc}}$ (by local Rellich compactness). But from the uniform decay estimate \eqref{ineq:un_unif} we actually deduce that $u_n \to u_*$ in $L^{\alpha+2}$, which implies that $u_* \not \equiv 0$. Finally, by weak convergence, notice that $\| (-\DD)^{s/2} u_* \|_{L^2} \leq 1$ and $\| u_* \|_{L^2} \leq 1$. Thus, we find
\be
\alpha = \lim_{n \to \infty} J(u_n) = \frac{1}{\| u_* \|_{L^{\alpha+2}}^{\alpha+2}} \geq  J(u_*) \geq \alpha.
\ee
It follows that $u_* \geq 0$ and $u_* \not \equiv 0$ is a nonnegative  minimizer for $J(u)$. Moreover, since equality holds everywhere, we note that we must have $\| (-\DD)^{s/2} u_* \|_{L^2} = \| u_* \|_{L^2}=1$. Hence we also have strong convergence $u_n \to u_*$ in $H^s$. To complete the proof of (i), we note that the minimizer $u_*$ satisfies $\partial_{\eps =0} J(u + \eps \varphi)=0$ for all $\varphi \in C^{\infty}_0$. A calculation shows that the function $Q(\cdot) = \mu u_*(\lambda \cdot)$ solves \eqref{eq:Q} if the scaling parameters $\mu > 0$ and $\lambda >0$ are suitably chosen. Note that $Q \in H^s$ is also a nonnegative minimizer for $J(u)$.

We now sketch the proof of part (ii) by using the results from the literature. Let $Q \in H^s(\R^N)$ with $Q \geq 0$ and $Q\not \equiv 0$ solve \eqref{eq:Q}. By following the arguments in \cite{MaZh-10}, we deduce that 
$$
\mbox{$Q(x-x_0)$ is radial, positive, and strictly decreasing in $|x-x_0|$},
$$
where $x_0 \in \R^N$ is some translation. Indeed, we only have to verify that the kernel $K=K(x-y)$ of the resolvent $((-\DD)^s + 1)^{-1}$ on $\R^N$ satisfies the following properties: 1.) $K = K(|z|)$ is real-valued and radial, 2.) $K(|z|) > 0$ is strictly positive for $z \in \R^N$, and 3.) $K(|z|)$ is monotone decreasing in $|z|$. In fact, all these properties hold true in our case, as we readily see from Lemma \ref{lem:Greenunif}. Hence we conclude from the moving plane arguments in \cite{MaZh-10} that $Q(x-x_0)$ is radial, positive, and (but not necessarily strictly) decreasing in $|x-x_0|$. To show that $Q(x-x_0)$ is indeed strictly decreasing in $|x-x_0|$, we argue as follows. Without loss of generality, we can assume that $x_0=0$ and thus $Q(x) = Q(|x|) > 0$ holds. Since $\partial_{x_i} Q \in \mathrm{ker} \, L_+$, where $L_+ = (-\DD)^s + 1 - (\alpha+1) Q^\alpha$, we deduce that
\be
L_{+,1} Q' =  0,
\ee
with the notation used in Section \ref{sec:nondeg}. We have that $Q'(r) \leq 0$, since $Q(r)$ is monotone decreasing. By Lemma \ref{lem:perron}, we conclude that $Q'(r)$ is (up to a sign) the unique ground state of $L_{+,1}$. Therefore, we have that either $Q'(r) < 0$ or $Q'(r) > 0$ for $r > 0$, where the first possibility is clearly ruled out. Hence $Q'(r) > 0$ for all $r > 0$, which shows that $Q(r)$ is strictly decreasing.

To show that $Q \in H^{2s+1}(\R^N)$ holds, we can simply follow the arguments in \cite{FrLe-10} where the case $N=1$ is considered. The smoothness $Q \in C^\infty$ in $\R^N$ follows from localizing the equation \eqref{eq:Q} on any open ball $B_R$ of fixed radius $R >0$. By the strict positivity $Q=Q(x-x_0) >0$, we see that $f : B_R \to \R$ with $f(x) = Q^{\alpha+1}(x)$ is a smooth function (with bounds depending on $R>0$). A bootstrap argument shows that $Q \in C^\infty$ on any ball $B_R$. 

Finally, we prove the lower and upper pointwise bounds for $Q$ stated in Proposition \ref{prop:Q}. First, we claim that $Q \in L^\infty$. (This is obvious if $s > N/2$ by Sobolev embedding.) Indeed, this follows from the $L^p$-properties for the resolvent kernel in Lemma \ref{lem:Greenunif} (ii). Using Young's inequality, we can iterate the identity $Q=((-\DD)^s + 1)^{-1} Q^{\alpha+1}$ finitely many times to conclude that $Q \in L^\infty$ holds. Next, by Proposition \ref{prop:regu}, we find that $Q \in C^{0,\beta}$ for any $\beta < 2s$. Since $Q \in L^2$, this shows that $Q$ vanishes at infinity. Now we note that $HQ = -Q$ with $H=(-\DD)^s + V$ with $V =- Q^{\alpha} \in L^\infty$ and $V \to 0$ as $|x| \to +\infty$. Hence we can apply Lemma \ref{lem:eigendecay} to find the upper bound $Q(x) \leq C \langle x \rangle^{-N-2s}$. Moreover, by modifying the arguments in the proof of Lemma \ref{lem:eigendecay} and using that $Q(x) > 0$ is positive, we also obtain the lower bound $Q(x) \geq C \langle x \rangle^{-N-2s}$. \end{proof}

\end{appendix}

\bibliographystyle{siam}
\bibliography{fractionalbib.bib}

\end{document}